\documentclass[reqno]{amsart}
\usepackage[utf8]{inputenc}
\usepackage[english]{babel}

\usepackage{amsfonts,amsthm,amssymb,mathtools}
\usepackage[alphabetic,abbrev]{amsrefs}
\numberwithin{equation}{section}

\usepackage[unicode]{hyperref}
\usepackage[capitalise,noabbrev]{cleveref} 
\crefname{equation}{}{} 
\Crefname{equation}{}{}

\newtheorem{theorem}{Theorem}[section]
\newtheorem{lemma}[theorem]{Lemma}
\newtheorem{proposition}[theorem]{Proposition}

\theoremstyle{definition}
\newtheorem{definition}[theorem]{Definition}

\theoremstyle{remark}
\newtheorem{remark}[theorem]{Remark}

\newcommand{\N}{\mathbb{N}}
\newcommand{\Z}{\mathbb{Z}}

\newcommand{\R}{\mathbb{R}}
\newcommand{\C}{\mathbb{C}}

\newcommand{\g}{\mathfrak{g}}

\renewcommand{\S}{\mathcal{S}}

\DeclarePairedDelimiter\norm{\lVert}{\rVert}

\renewcommand{\Re}{\operatorname{Re}}
\renewcommand{\Im}{\operatorname{Im}}

\DeclareMathOperator{\supp}{supp}
\DeclareMathOperator{\id}{id}
\DeclareMathOperator{\ran}{ran}

\DeclareMathOperator{\chr}{char}

\DeclareMathOperator{\tr}{tr}

\newcommand{\tL}[2]{\Delta^{(#1)}_{\R^{#2}}}
\newcommand{\atL}[3]{\Delta^{\mathbf #1,\mathbf #2}_{\R^{#3}}}
\newcommand{\proj}[3]{\Pi_{\mathbf{#1}}^{\mathbf{#2},\mathbf{#3}}}
\newcommand{\Lag}[3]{\varphi_{\mathbf{#1}}^{\mathbf{#2},\mathbf{#3}}}
\newcommand{\eigv}[3]{\lambda_{\mathbf{#1}}^{\mathbf{#2},\mathbf{#3}}}
\newcommand{\eigvp}[2]{\lambda_{\mathbf{#1}}^{#2}}
\newcommand{\eigf}[4]{\Phi_{#1,#2}^{\mathbf{#3},\mathbf{#4}}}

\newcommand{\cluster}[3]{P_{#1}^{\mathbf{#2},\mathbf{#3}}}

\AtBeginDocument{%
   \def\MR#1{}
}

\begin{document}

\title[Restriction type estimates on two-step Lie groups]{Restriction type estimates on general two-step stratified Lie groups}
\author{Lars Niedorf}
\address{Department of Mathematics, University of Wisconsin-Madison,
480 Lincoln Dr, Madison, WI-53706, USA}
\email{niedorf@wisc.edu}
\date{December 12, 2024}
\thanks{The author gratefully acknowledges the support by the Deutsche Forschungsgemeinschaft (DFG) through grant MU 761/12-1.}

\begin{abstract}
We prove restriction type estimates for sub-Laplacians on general two-step stratified Lie groups. The core of our approach is to use spectral cluster estimates to effectively control the eigenvalue distribution of a family of anisotropic twisted Laplacians.
\end{abstract}

\subjclass[2020]{42B15, 22E25, 22E30, 43A85}
\keywords{Two-step stratified Lie group, sub-Laplacian, restriction type estimate, spectral cluster}

\maketitle

\section{Introduction}

\subsection{Restriction estimates}

Let $S\subseteq\R^n$ be a smooth hypersurface. The concept of Fourier restriction estimates was introduced by Stein in the seventies (for general sub-manifolds), which are a priori estimates of the form
\begin{equation}\label{eq:Fourier-p-q}
\|\hat f|_S\|_{L^q(S,\sigma)} \le C_p \norm{f}_{L^p(\R^n)} \quad \text{for all } f\in \mathcal S(\R^n)
\end{equation}
where $p,q\in [1,\infty]$, $\sigma$ is the surface measure on $S$, and $\mathcal S(\R^n)$ is the space of Schwartz functions.
For the sphere $S = S^{n-1} = \{ \xi \in\R^n : |\xi| = 1 \}$ and $n\ge 2$, the \textit{restriction conjecture} \cite{St79,Ta04} states that the restriction estimate \cref{eq:Fourier-p-q} holds if and only if
\[
1\le p < \frac{2n}{n+1} \quad \text{and}\quad 1\le q\le \frac{n-1}{n+1} p',
\]
where $p'=(1-1/p)^{-1}$ is the dual exponent of $p$. The restriction conjecture for the sphere is solved in dimension $n=2$, but is widely open for $n\ge 3$, although many partial results are available, including work by Bourgain \cite{Bo91}, Tao \cite{Ta03}, and Guth \cite{Gu16}. In the range $1\le p \le 2(n+1)/(n+3)$, due to the work of Stein and Tomas \cite{St86,To75} (see also \cite{St77,Gr81}), the restriction conjecture, which is then the estimate
\[
\|\hat f|_{S}\|_{L^2(S,\sigma)} \le C_p \norm{f}_{L^p(\R^n)} \quad \text{for all } f\in \mathcal S(\R^n),
\]
is known to hold for the sphere, and also for more general classes of surfaces with non-vanishing Gaussian curvature. (Note that  $(n-1)p'/(n+1) =2$ for the endpoint $p=2(n+1)/(n+3)$. The $L^p$-$L^q$ restriction estimate \cref{eq:Fourier-p-q} follows from the $L^p$-$L^2$ estimate via interpolation with the trivial $L^1$-$L^\infty$ estimate.)

It is by now well-known that restriction estimates are closely related to many other problems in harmonic analysis and PDE, such as spectral multiplier estimates and Bochner--Riesz summability for the Laplacian, the Kakeya conjecture, and local smoothing, see \cite{Ta99}.

Given these close connections for the Laplacian, one hopes that many of these will carry over to elliptic or even sub-elliptic differential operators. One class of operators that has received a lot of attention in the last few decades are second-order differential operators which are a sum of squares of the form
\begin{equation}\label{eq:sum-of-squares}
L = -(X_1^2+\dots+X_k^2),
\end{equation}
where $X_1,\dots,X_k$ are some vector fields on a smooth manifold $M$. These operators can be seen as a natural replacement for the Laplacian when passing to a sub-elliptic setting. Due to a celebrated theorem of Hörmander \cite{Hoe67}, the operator in \eqref{eq:sum-of-squares} is hypoelliptic whenever the iterated commutators
\[
X_{j_1},[X_{j_1},X_{j_2}],[X_{j_1},[X_{j_2},X_{j_3}]],\dots \quad \text{where } j_i\in\{1,\dots,k\}
\]
up to a certain order span the tangent space of the underlying manifold $M$ at every point. Rothschild and Stein \cite{RoSt76} pointed out that operators of the form \eqref{eq:sum-of-squares} satisfying Hörmander's bracket generating condition can be locally approximated by left-invariant sub-Laplacians on stratified Lie groups, after the initial vector fields $X_1,\dots,X_k$ on $M$ are lifted to vector fields $\smash{\tilde X_1},\dots,\smash{\tilde X_k}$ on some larger space~$\smash{\tilde M}$ by some freeing procedure.

When asking for analogs of restriction estimates in the setting of stratified Lie groups, things change dramatically compared to the Euclidean setting. For instance, for the sub-Laplacian on the Heisenberg group, the only possibility of having restriction estimates is the case where $p=1$ (unless one passes to mixed $L^p$-norms), which is due to the fact that the Heisenberg group admits a one-dimensional center \cite{Mue90}. However, the situation improves for left-invariant sub-Laplacians on Heisenberg type groups with center of dimension $d_2>1$, where it is possible to prove restriction estimates for the Stein--Tomas range $1\le p\le 2(d_2+1)/(d_2+3)$, see \cite{Th91,LiZh11}, but it is not yet known if such restriction estimates hold beyond the class of Heisenberg type groups. Attempts to prove restriction estimates for the larger class of Métivier groups where the group satisfies a certain non-degeneracy condition have unfortunately failed so far (see the remarks in \cite{Ca18} and \cite{Ni23} on \cite{CaCi13,LiZh18}).

\subsection{Statement of the main results}

The purpose of this paper is to establish restriction \textit{type} estimates for the whole class of two-step stratified Lie groups.

Before stating the main results, we need to introduce some notation. Let $G$ be a two-step stratified Lie group, that is, a connected, simply connected nilpotent Lie group whose Lie algebra $\g$ (which is the tangent space at the identity of $G$) admits a decomposition $\g=\g_1\oplus\g_2$ into two non-trivial subspaces $\g_1,\g_2\subseteq\g$, where $[\g_1,\g_1]=\g_2$ and $\g_2\subseteq\g$ is contained in the center of $\g$. In the following, we refer to $\g_1$ and $\g_2$ as being the \textit{first} and \textit{second layer} of $\g$, respectively. Let
\begin{equation}\label{eq:dimensions}
d_1=\dim \g_1\ge 1,\quad d_2=\dim \g_2\ge 1\quad \text{and}\quad d=\dim \g.
\end{equation}
Given a basis $X_1,\dots,X_{d_1}$ of the first layer $\g_1$, we identify each element of the basis by a left-invariant vector field on $G$ via the Lie derivative, and consider the associated sub-Laplacian $L$, which is the second-order differential operator
\begin{equation}\label{def:sub-Laplacian}
L = -(X_1^2+\dots+X_{d_1}^2).
\end{equation}
We also choose a basis $U_1,\dots,U_{d_2}$ of the second layer $\g_2$. Let $\langle\cdot,\cdot\rangle$ be the inner product rendering $X_1,\dots,X_{d_1},U_1,\dots,U_{d_2}$ an orthonormal basis of $\g$. The inner product $\langle\cdot,\cdot\rangle$ induces a norm on the dual $\g_2^*$ which we denote by $|\cdot|$. For $\mu\in\g_2^*$, let $J_\mu$ be the skew-symmetric endomorphism such that
\begin{equation}\label{eq:skew-form-ii}
\mu([x,x']) =\langle J_\mu x,x'\rangle,\quad x,x'\in\g_1.
\end{equation}
Then $G$ is a \textit{Heisenberg type group} if the endomorphisms $J_\mu$ are orthogonal for all $\mu\in \g_2^*$ of length 1, which means that
\[
J_\mu^2 = - |\mu|^2 \id_{\g_1}\quad\text{for all } \mu\in \g_2^*.
\]

The restriction type estimates in \cref{thm:restriction-type} are stated in terms of the norms 
\[
\norm{F}_{M,2} = \bigg(\frac{1}{M} \sum_{K\in\Z} \,\sup _{\lambda \in [\frac{K-1}{M}, \frac{K}{M})}|F(\lambda)|^2\bigg)^{1 / 2},\quad M\in(0,\infty)
\]
which were introduced by Cowling and Sikora in \cite{CoSi01}. We introduce an additional truncation along the spectrum of the operator
\[
U=(-(U_1^2+\dots+U_{d_2}^2))^{1/2}.
\]
Due to this additional truncation, we also refer to the corresponding restriction type estimates as \textit{truncated} restriction type estimates. A similar truncation is used in \cite{Ni23} for sub-Laplacians in Heisenberg type groups, and in \cite{Ni22} in the related setting of Grushin operators.


Our truncated restriction type theorem reads as follows.
Given any Euclidean space of dimension $n\in\N\setminus\{0\}$, we denote by
\[
p_n := \frac{2(n+1)}{n+3}
\]
the \textit{Stein--Tomas exponent} associated with that space.

\begin{theorem}\label{thm:restriction-type}
Let $G$ be a two-step stratified Lie group, and, as in \cref{def:sub-Laplacian}, let $L$ be a sub-Laplacian on $G$. Suppose that $1\le p\le \min\{p_{d_1},p_{d_2}\}$ with dimensions $d_1,d_2$ as in \eqref{eq:dimensions}. If $F:\R\to\C$ is a bounded Borel function supported in a compact subset $A\subseteq (0,\infty)$ and $\chi:(0,\infty)\to\C$ is a smooth function with compact support, then
\begin{equation}\label{eq:intro-restriction}
\norm{ F(L)\chi(2^\ell U) }_{p\to 2}
\le C_{A,p,\chi} 2^{-\ell d_2(\frac 1 p - \frac 1 2)} \|F\|_2^{1-\theta_p} \norm{F}_{2^{\ell},2}^{\theta_p} \quad \text{for all }\ell\in\Z,
\end{equation}
where $\theta_p\in [0,1]$ satisfies $1/p = (1-\theta_p) + \theta_p/\min\{p_{d_1},p_{d_2}\}$.
\end{theorem}

As we will see later in \cref{rem:ell_0}, one has $F(L)\chi(2^\ell U)=0$ for $\ell$ being small enough in \eqref{eq:intro-restriction}, so  \cref{thm:restriction-type} should actually be read as a statement for all $\ell\in [-\ell_0,\infty)$, where $\ell_0\in\N$ is a constant depending on the matrices $J_\mu$, the inner product $\langle\cdot,\cdot\rangle$ on $\g$, and the compact subset $A\subseteq (0,\infty)$ above. 

Although the above theorem is valid for all two-step stratified Lie groups, the presence of the norm $\norm{F}_{2^{\ell},2}$ in \eqref{eq:intro-restriction} imposes a slight drawback. By (3.19) and (3.29) of \cite{CoSi01} (or alternatively Lemma 3.4 of \cite{ChHeSi16}), for $s>1/2$, the norm $\Vert\cdot\Vert_{M,2}$ can be estimated by
\begin{equation}\label{eq:norm-1}
\norm{F}_{L^2}\le \norm{F}_{M,2} \le C_s \big( \norm{F}_{L^2} + M^{-s} \norm{F}_{L^2_s} \big),
\end{equation}
which means that $\Vert\cdot\Vert_{M,2}$ is stronger than the $L^2$-norm. Ignoring for a moment the additional truncation, the restriction type estimate \eqref{eq:intro-restriction} with $\norm{F}_{2^{\ell},2}$ replaced by $\norm{F}_{L^2}$ would be equivalent (see for instance \cite[Proposition 4.1]{SiYaYa14}) to a restriction type estimate for the Strichartz projectors $\mathcal P_\lambda$, which are formally given by $\mathcal P_\lambda=\delta_\lambda(L)$, where $\delta_\lambda$ is the Dirac delta distribution at $\lambda\in\R$. This would recover some of the results of the (erroneous) article \cite{CaCi13}. However, in its current form, \eqref{eq:intro-restriction} is not sufficient to push restriction estimates beyond the class of Heisenberg type groups.

On the other hand, the truncated restriction type estimates can still be used to prove spectral multiplier estimates for the whole class of Métivier groups (which is the class of two-step stratified Lie groups where the matrices $J_\mu$ are invertible for all $\mu\in\g_2^*\setminus\{0\}$), by exploiting the fact that the dimension $d_1$ of the first layer is in general much larger than the dimension $d_2$ of the second layer if $G$ is a Métivier group. This is done in a follow-up paper.

\subsection{Structure of the paper}

In \cref{sec:sketch}, we briefly sketch the proof of \cref{thm:restriction-type}. In \cref{sec:spectral-theory}, we analyze the spectral decomposition of the sub-Laplacian $L$ and the vector fields $-iU_1,\dots,-iU_{d_2}$, which admit a joint functional calculus. In particular, we show that the sub-Laplacian $L$ corresponds to a family of anisotropic twisted Laplacians by a partial Fourier transform along the second layer.

In \cref{sec:cluster}, we derive spectral cluster estimates for those anisotropic twisted Laplacians, which are then subsequently exploited in \cref{sec:restriction} to prove the restriction type estimates of \cref{thm:restriction-type}.

\subsection{Notation}

We let $\N=\{0,1,2,\dots\}$. The indicator function of a subset $A$ of some measurable space will be denoted by $\mathbf{1}_A$. We write $A\lesssim B$ if $A\le C B$ for a constant $C$. If $A\lesssim B$ and $B\lesssim A$, we write $A\sim B$. Given two suitable functions $f$ and $g$ on a two-step stratified Lie group $G$, let $f*g$ denote their group convolution given by
\[
f*g(x,u) = \int_{G} f(x',u')g\big((x',u')^{-1}(x,u)\big) \,d(x',u'),\quad (x,u)\in G,
\]
where $d(x',u')$ denotes the Lebesgue measure on $G$. The space of Schwartz functions on $\R^n$ will be denoted by $\mathcal S(\R^n)$. For $s\ge 0$ and $q\in [1,\infty]$, we denote by $L^q_s(\R)\subseteq L^q(\R)$ the Sobolev space of fractional order $s$.

\subsection{Acknowledgments}

I am deeply grateful to my advisor Detlef Müller for his unwavering support and many fruitful discussions about the subject of this work. I also wish to express my thanks to Alessio Martini for kindly hosting me for a week at the Mathematics Department of the Politecnico di Torino and for the many mathematical discussions during my visit.


\section{Sketch of the proof} \label{sec:sketch}

As in \cite{LiZh11,ChOu16,Ni22,Ni23}, the proof of the truncated restriction type estimate in \cref{thm:restriction-type} relies on combining two restriction (type) estimates, namely an $L^p$-$L^2$ estimate for the spectral projections associated with the twisted Laplacian on the first layer, and a restriction estimate for the sphere on the second layer, whence we require $1\le p\le \min\{p_{d_1},p_{d_2}\}$ as a condition for the range of $p$. More precisely, conjugating the sub-Laplacian $L$ by the partial Fourier transform given by
\[
f^\mu(x) = \int_{\g_2} f(x,u) e^{- i \langle\mu, u\rangle} \, du,\quad x\in \g_1,\mu\in \g_2^*
\]
transforms the sub-Laplacian $L$ into a family $(L^\mu)_{\mu\in\g_2^*}$ of twisted Laplacians $L^\mu$ on the first layer $\g_1$. Assuming here for the sake of simplicity that $G$ is a Métivier group, then, at least on some non-empty Zariski-open subset $\g_{2,r}^*\subseteq \g_2^*$, each twisted Laplacian $L^\mu$ admits an orthonormal basis of eigenfunctions associated with the eigenvalues
\[
\eigv{k}{b^\mu}{r} = \sum_{n=1}^N \left(2k_n+r_n\right)b_n^\mu,\quad \textbf{k}=(k_1,\dots,k_N)\in\N^N
\]
where $N\in\N\setminus\{0\}$ and $\textbf{r}=(r_1,\dots,r_N)\in (\N\setminus\{0\})^N$ are fixed parameters independent of $\mu$, and $\mu\mapsto \textbf b^\mu=(b_1^\mu,\dots,b_N^\mu)\in (0,\infty)^N$ is a function homogeneous of degree~1, which is smooth on $\g_{2,r}^*$ and extends to a continuous function on $\g_2^*$.

The special case where $G$ is a Heisenberg type group corresponds to taking $N=1$ and $\textbf{b}^\mu=|\mu|$. Then the eigenvalues of the twisted Laplacian $L^\mu$ are given by $|\mu|[k]$, $k\in\N$, where $[k]:=2k+d_1/2$. In \cite{LiZh11,Ni23} (and similar in \cite{ChOu16,Ni22}), the proofs of the restriction type estimates there rely on $L^p$-$L^2$ spectral projection estimates of the form
\begin{equation}\label{eq:intro-projection}
\|\mathbf{1}_{\{|\mu|[k]\}}(L^\mu)\|_{L^p(\g_1)\to L^2(\g_1)} \le C_p \left|\mu\right|^{\frac {d_1} 2 (\frac 1 p - \frac 1 2)} [k]^{\frac {d_1} 2 (\frac 1 p - \frac 1 2)-\frac 1 2}, \quad k\in \N,
\end{equation}
where $\mathbf{1}_{\{|\mu|[k]\}}:\R\to \{0,1\}$ denotes the indicator function of the point $|\mu|[k]$.

Writing $\mu\in \g_2^*$ in polar coordinates, that is, $\mu=\rho \omega$, where $\rho\in [0,\infty)$ and $\omega\in S^{d_2-1}$, note that $\textbf{b}^\mu$ does not depend on $\omega$ in the Heisenberg type case. To adapt the approach of \cite{Ni23} to setting of arbitrary two-step Lie groups and handle the dependence on the parameter $\omega$, we employ spectral cluster estimates of the form
\begin{equation}\label{eq:intro-cluster}
\norm{\mathbf{1}_{[K|\mu|,(K+1)|\mu|)}(L^\mu) }_{L^p(\g_1)\to L^2(\g_1)} \le C_p \left|\mu\right|^{\frac {d_1} 2 (\frac 1 p - \frac 1 2)} (K+1)^{\frac {d_1} 2 (\frac 1 p - \frac 1 2)-\frac 1 2}
\end{equation}
for $K\in\N$. These spectral cluster estimates are inspired by the work of Sogge \cite[Chapter 5]{So93}. Although the spectral projection estimates \eqref{eq:intro-projection} and the spectral cluster estimates \eqref{eq:intro-cluster} are equivalent in the case of Heisenberg type groups, using spectral cluster estimates in the setting of arbitrary two-step Lie groups leads to restriction type estimates in terms of the Cowling--Sikora norms $\|\cdot\|_{M,2}$ in place of the $L^2$-norm.

\section{Sub-Laplacians on two-step stratified Lie groups} \label{sec:spectral-theory}

\subsection{Joint functional calculus}

Let $G$ be a two-step stratified Lie group. Then its Lie algebra $\g$, which is the tangent space $T_e G$ at the identity $e\in G$, admits a decomposition $\g=\g_1\oplus\g_2$, where $[\g_1,\g_1]=\g_2$ and $\g_2\subseteq \g$ is contained in the center of~$\g$. Using exponential coordinates, we identify $G$ with its Lie algebra $\g$. The group multiplication is then given by
\[
(x,u)(x',u')=\left(x+x',u+u'+\tfrac 1 2 [x,x']\right),\quad x,x'\in \g_1,u,u'\in \g_2.
\]
We choose bases $X_1,\dots,X_{d_1}$ and $U_1,\dots,U_{d_2}$ of $\g_1$ and $\g_2$, respectively. By means of these bases, we identify $\g_1\cong\R^{d_1}$ and $\g_2\cong\R^{d_2}$. Let $\langle \cdot,\cdot\rangle$ denote the inner product with respect to which $X_1,\dots,X_{d_1},U_1,\dots,U_{d_2}$ becomes an orthonormal basis of $\g$. As usual, the Lie algebra $\g$ will also be identified with the Lie algebra of smooth left-invariant vector fields on $G$ via the Lie derivative. Then, given a smooth function $f$ on $G$, we have
\begin{align*}
X_j f(x,u)
& = \frac{d}{dt} f\big((x,u)(tX_j,0)\big)\big|_{t=0} \\
& = \partial_{x_j} f(x,u) + \frac 1 2 \sum_{k=1}^{d_2}\left\langle U_k,[x,X_j]\right\rangle \partial_{u_k} f(x,u),\\
U_k f(x,u)
& = \partial_{u_k} f(x,u).
\end{align*}
The sub-Laplacian $L$ associated with the vector fields $X_1,\dots,X_{d_1}$ is given by
\[
L = -\left(X_1^2+\dots+X_{d_1}^2\right).
\]

For $f\in L^1(G)$ and $\mu\in\g_2^*$, let $f^\mu$ denote the $\mu$-section of the partial Fourier transform along the second layer $\g_2$ given by
\[
f^\mu(x) = \int_{\g_2} f(x,u) e^{- i \langle\mu, u\rangle} \, du,\quad x\in \g_1.
\]
Up to some constant, this defines an isometry $\mathcal F_2:L^2(\g_1\times \g_2)\to L^2(\g_1\times \g_2^*)$. Given $f\in L^2(G)$, we also write $f^\mu=(\mathcal F_2 f)(\cdot,\mu)$ (for almost all $\mu\in\g_2^*$) in the following.
For $f\in \mathcal S(G)$, we have $(X_j f)^\mu = X_j^\mu f^\mu$, where\[
X_j^\mu = \partial_{x_j} + \tfrac i 2 \omega_\mu(x,X_j),
\]
where $\omega_\mu$ denotes the bilinear form given by
\[
\omega_\mu(x,x') = \mu([x,x']),\quad x,x'\in\g_1.
\]
Moreover, let $L^\mu$ be the second order differential operator defined by
\[
L^\mu = -\left((X_1^\mu)^2+\dots+(X_{d_1}^\mu)^2\right).
\]
We call $L^\mu$ the \textit{$\mu$-twisted Laplacian} on $\g_1$. Direct computation shows that
\[
(Lf)^\mu=L^\mu f^\mu
 = \left(-\Delta_x + \tfrac 1 4 |J_\mu x|^2 - i\, \omega_\mu(x,\nabla)\right)f^\mu,
\]
where $J_\mu$ is again the endomorphism given by
\[
\langle J_\mu x,x' \rangle = \omega_\mu(x,x'),\quad x,x'\in\g_1,
\]
and $\omega_\mu(x,\nabla)$ is a short-hand notation for the operator
\[
\omega_\mu(x,\nabla) = \sum_{j=1}^{d_1}\omega_\mu(x,X_j)\partial_{x_j}.
\]

The operators $L,-iU_1,\dots,-iU_{d_2}$ form a system of formally self-adjoint, left-invariant and pairwise commuting differential operators, whence they admit a joint functional calculus \cite{Ma11}. Let $\textbf{U}$ be the vector of differential operators
\[
\mathbf U = (-iU_1,\dots,-iU_{d_2}).
\]
Since the joint functional calculus is compatible with unitary representation theory, the $\mu$-sections of the partial Fourier transform and the joint functional calculus of $L$ and $\textbf{U}$ are compatible as well.

\begin{proposition}\label{prop:joint-calculus}
If $F:\R\times\R^{d_2}\to\C$ is a bounded Borel function, then
\begin{equation}\label{eq:joint-calculus}
\left(F(L,\mathbf U)f\right)^\mu = F(L^\mu,\mu) f^\mu
\end{equation}
for all $f\in L^2(G)$ and almost all $\mu\in\g_2^*$.
\end{proposition}

\begin{proof}
This follows from the arguments of \cite[Section 1]{Mue90}. Let $\mathcal U(L^2(\g_1))$ denote the group of unitary operators on $L^2(\g_1)$. Applying Proposition 1.1 of \cite{Mue90} to the unitary representation $\pi_\mu:G\to\mathcal U(L^2(\g_1))$ given by
\[
\big(\pi_\mu(x,u)\varphi\big)(x')=e^{-i\langle \mu, u\rangle - \frac i 2 \omega_\mu (x,x') }\varphi(x'-x),\quad (x,u)\in G, x' \in \g_1
\]
yields \cref{eq:joint-calculus} if $F:\R\times\R^{d_2}\to\C$ is of the form $F(\lambda,\mu)=G(\lambda)H(\mu)$, where $\lambda\in\R$ and $\mu\in\R^{d_2}$. By a standard density argument, we obtain \cref{eq:joint-calculus} for any bounded Borel function $F$.
\end{proof}

\subsection{Decomposition into blocks of twisted Laplacians}\label{subsec:decomposition}

Next, we transform the twisted Laplacian $L^\mu$ into a more accessible form by choosing an appropriate basis with respect to which $L^\mu$ turns into a differential operator consisting of blocks of rescaled twisted Laplacians with standard symplectic form.

\begin{definition}\label{def:classical-twisted}
For $m\in\N\setminus\{0\}$ and $\lambda>0$, we call the operator
\[
\tL{\lambda}{2m} = -\Delta_z + \tfrac 1 4 \lambda^2 |z|^2 - i \lambda \omega(z,\nabla_z),\quad z \in \R^{2m}
\]
the \textit{$\lambda$-twisted Laplacian} on $\R^{2m}$, where
\begin{equation}\label{eq:standard-symplectic-1}
\omega(z,w)=\omega_{\R^{2m}}(z,w)=\langle Jz, w\rangle_{\R^{2m}}
\end{equation}
is the standard symplectic form induced by the $2m\times 2m$ matrix
\[
J = J_{\R^{2m}} = \begin{pmatrix}
0 & -\mathrm{id}_{\R^{m}} \\
\mathrm{id}_{\R^{m}} & 0 
\end{pmatrix}.
\]
\end{definition}

\begin{definition}
Let $d_1\in\N\setminus\{0\}$. Given parameters $\mathbf b=(b_1,\dots,b_N) \in [0,\infty)^N$ and $\mathbf r=(r_1,\dots,r_N)\in(\N\setminus\{0\})^N$ with $N\in\N\setminus\{0\}$ and $2\left|\textbf{r}\right|_1\le d_1$, we call
\[
\atL{b}{r}{d_1} =
(-\Delta_{\R^{r_0}}) \oplus \tL{b_1}{2r_1} \oplus \dots  \oplus \tL{b_N}{2r_N}
\]
the \textit{anisotropic twisted Laplacian of type $(\mathbf b,\mathbf r)$} on $\R^{d_1}=\R^{r_0}\oplus \R^{2r_1}\oplus\dots\oplus\R^{2r_N}$, where $r_0=d_1-2\left|\textbf{r}\right|_1=d_1-2\left(r_1+\dots+r_N\right)$ and $\Delta_{\R^{r_0}}$ is the Euclidean Laplacian on $\R^{r_0}$.
\end{definition}

\begin{proposition}\label{prop:rotation}
There exist a non-empty, homogeneous Zariski-open subset $\g_{2,r}^*$ of $\g_2^*$, numbers $N\in\N\setminus\{0\}$, $r_0\in\N$, $\mathbf r=(r_1,\dots,r_N)\in(\N\setminus\{0\})^N$, a function $\mu \mapsto \mathbf{b}^\mu = (b^\mu_1,\dots,b^\mu_N)\in [0,\infty)^N$ on $\g_2^*$, functions $\mu\mapsto P_n^\mu$ on $\g_{2,r}^*$ with $P_n^\mu:\g_1\to\g_1$, $n\in\{1,\dots,N\}$, and a function $\mu\mapsto R_\mu\in O(d_1)$ on $\g_{2,r}^*$ such that
\begin{equation}\label{eq:spectral-decomp-0}
-J_\mu^2 = \sum_{n=1}^N (b_n^\mu)^2 P_n^\mu \quad \text{for all } \mu \in \g_{2,r}^*,
\end{equation}
with $P_n^\mu R_\mu=R_\mu P_n$, $J_\mu(\ran P_n^\mu)\subseteq \ran P_n^\mu$ for the range of $P_n^\mu$ and
\begin{equation}\label{eq:symplectic-0}
\omega_\mu(P_n^\mu x,P_n^\mu x')=b_n^\mu \, \omega_{\R^{2r_n}}(P_n R_\mu^{-1}x,P_n R_\mu^{-1}x')\quad \text{for all } x,x'\in \g_1
\end{equation}
for all $\mu \in \g_{2,r}^*$ and all $n\in\{0,\dots,N\}$, where $P_n$ denotes the projection from $\R^{d_1}=\R^{r_0}\oplus\R^{2r_1}\oplus \dots\oplus \R^{2r_N}$ onto the $n$-th layer, and, for all bounded Borel functions $F:\R\to\C$ and all $\phi\in L^2(\g_1)$,
\begin{equation}\label{eq:rotation}
(F(L^\mu)\phi)\circ R_\mu = F\big(\atL{b^\mu}{r}{d_1}\big)(\phi\circ R_\mu)\quad \text{in } L^2(\R^{d_1})
\end{equation}
for almost all $\mu\in\g_{2,r}^*$, where
\begin{itemize}
\item[(i)] the functions $\mu \mapsto b^\mu_n$ are homogeneous of degree~$1$ and continuous on $\g_2^*$, real analytic on $\g_{2,r}^*$, and satisfy $b_n^\mu > 0$ for all $\mu \in \g_{2,r}^*$ and $n \in \{1,\dots,N\}$, and $b_n^\mu \neq b_{n'}^\mu$ if $n \neq n'$ for all $\mu \in \g_{2,r}^*$ and $n,n' \in \{1,\dots,N\}$,
\item[(ii)] the functions $\mu\mapsto P_n^\mu$ are (componentwise) real analytic on $\g_{2,r}^*$, homogeneous of degree $0$, and the maps $P_n^\mu$ are orthogonal projections on $\g_1$ of rank $2r_n$ for all $\mu \in \g_{2,r}^*$, with pairwise orthogonal ranges,
\item[(iii)] $\mu\mapsto R_\mu$ is a Borel measurable function on $\g_{2,r}^*$ which is homogeneous of degree~$0$, and there is a family $(U_\ell)_{\ell\in\N}$ of disjoint Euclidean open subsets $U_\ell\subseteq \g_{2,r}^*$ whose union is $\g_{2,r}^*$ up to a set of measure zero such that $\mu\mapsto R_\mu$ is (componentwise) real analytic on each $U_\ell$.
\end{itemize}
\end{proposition}

\begin{remark}
Given an anisotropic twisted Laplacian $\atL{b}{r}{d_1}$, the parameters $\textbf{b}$ and $\textbf{r}$ are clearly not unique, but they are, up to a permutation of the blocks, if one additionally requires $b_n>0$ and $b_n\neq b_{n'}$ if $n\neq n'$ for all $n,n'\in\{1,\dots,N\}$. However, in property (i) above, it may happen that $b_n^\mu = 0$ or $b_n^\mu = b_{n'}^\mu$ for $n\neq n'$ if $\mu$ lies in the Zariski closed set $\g_2^*\setminus\g_{2,r}^*$.
\end{remark}

\begin{proof}
Since $J_\mu$ is skew-symmetric, the endomorphism $-J_\mu^2=J_\mu^* J_\mu$ is self-adjoint and non-negative. Let $p_\mu(\lambda) = \det(\lambda+J_\mu^2)$ be the characteristic polynomial of $-J_\mu^2$. Then, by \cite[Lemma 4]{MaMue14b}, there exists a non-empty, homogeneous Zariski-open subset $\g_{2,r}^*$ of $\g_2^*$ and numbers $N\in\N\setminus\{0\}$, $r_0\in\N$, $r_1,\dots,r_N\in \N\setminus\{0\}$ such that
\[
p_\mu(\lambda)=\lambda^{r_0} (\lambda-(b_1^\mu)^2)^{2r_1} \cdots (\lambda-(b_N^\mu)^2)^{2r_N}
\]
for all $\mu\in\g_{2,r}^*$, with functions $\mu\mapsto b_n^\mu$ that satisfy the properties of (i). Note that $d_1=r_0+2r_1+\dots+2r_N$. As in \cite[Lemma 5]{MaMue14b}, the factorization of the characteristic polynomial yields
\begin{equation}\label{eq:spectral-decomp}
-J_\mu^2 = \sum_{n=1}^N (b_n^\mu)^2 P_n^\mu \quad \text{for all } \mu \in \g_{2,r}^*,
\end{equation}
where the $P_n^\mu$ are orthogonal projections on $\g_1$ of rank $2r_n$ for all $\mu \in \g_{2,r}^*$, with pairwise orthogonal ranges. The $P_n^\mu$ are real analytic functions of $\mu \in \g_{2,r}^*$, which are homogeneous of degree $0$. From the spectral decomposition \eqref{eq:spectral-decomp}, one deduces that $P_n^\mu=F_{n,\mu}(-J_\mu^2)$ for any (Borel) function such that $F_{n,\mu}(0)=0$ and $F_{n,\mu}((b_{n'}^\mu)^2)=\delta_{n,n'}$ for all $n\in\{1,\dots,N\}$. Choosing for instance $F_{n,\mu}$ as an interpolation polynomial shows that $J_\mu$ and all the projections $P_n^\mu=F_{n,\mu}((-iJ_\mu)^2)$ commute. Thus, $J_\mu(\ran P_n^\mu)\subseteq \ran P_n^\mu$ for the range of $P_n^\mu$, and by \eqref{eq:spectral-decomp},
\begin{equation}\label{eq:H-type-block}
-\big(J_\mu\big|_{\ran P_n^\mu}\big)^2 = (b_n^\mu)^2 \id_{\ran P_n^\mu}.
\end{equation}
Let $J_n^\mu$ denote the restriction of $J_\mu$ onto $\ran P_n^\mu$. We consider $J_n^\mu$ as an operator on the complexification $(\ran P_n^\mu)^\C$ of $\ran P_n^\mu$. Then $iJ_n^\mu$ is self-adjoint and according to \cref{eq:H-type-block}, it admits only the eigenvalues $\pm b_n^\mu$. The spectral projections associated with these two eigenvalues are given by
\[
P_{n,\pm}^\mu=\frac 1 2 \left( \id_{\ran P_n^\mu} \pm \,i\, (b_n^\mu)^{-1} J_n^\mu \right).
\]
Note that the functions $\mu\mapsto P_{n,\pm }^\mu$ are real analytic on $\g_{2,r}^*$. We take some arbitrary basis $v_1,\dots,v_{d_1}$ of $\g_1$ and consider for each $n\in\{1,\dots,N\}$ the set of vectors
\[
P_{n,+}^\mu v_1,\dots,P_{n,+}^\mu v_{d_1} \in (\ran P_n^\mu)^\C .
\]
We may find for any point $\mu_0\in \g_{2,r}^*$ a neighborhood $U_{\mu_0}\subseteq \g_{2,r}^*$ and a set of indices $I_{n,\mu_0}\subseteq \{1,\dots,d_1\}$ with $|I_{n,\mu_0}|=r_n$ such that the vectors $P_{n,+}^\mu v_j$, $j\in I_{n,\mu_0}$ are linearly independent. We denote these vectors (whose choice depends on $\mu_0$) by $w_{n,1}^\mu,\dots,w_{n,r_n}^\mu$. Then the maps $\mu\mapsto w_{n,m}^\mu$ are (componentwise) real analytic functions on $U_{\mu_0}$. Moreover, due to the homogeneity of the projections $P_{n,\pm}^\mu$, the functions $\mu\mapsto w_{n,m}^\mu$ are homogeneous of degree 0. Now applying the Gram-Schmidt process to the basis $w_{n,1}^\mu,\dots,w_{n,r_n}^\mu$ yields an orthonormal basis $\tilde w_{n,1}^\mu,\dots,\tilde w_{n,r_n}^\mu$ of the eigenspace associated with the eigenvalue $b_n^\mu$ of $i J_n^\mu$. Together with the complex conjugates of the orthonormal basis $\tilde w_{n,1}^\mu,\dots,\tilde w_{n,r_n}^\mu$, we obtain an orthonormal basis $v_{n,1}^\mu,\dots,v_{n,2r_n}^\mu$ of $\ran P_n^\mu$ such that
\begin{equation}\label{eq:symplectic}
\omega_\mu(v_{n,m}^\mu ,v_{n,m'}^\mu ) = b_n^\mu \, \omega_{\R^{2r_n}}(e_m,e_{m'})\quad \text{for all } m,m' \in \{1,\dots,2r_n\},
\end{equation}
where $e_m$ is the $m$-th standard basis vector of $\R^{2r_n}$ and $\omega_{\R^{2r_n}}$ denotes the standard symplectic form \cref{eq:standard-symplectic-1}. Note that $\omega_\mu(v_{n,m}^\mu ,v_{n',m'}^\mu )=0$ for $n\neq n'$ by construction. Let $P_0^\mu:=\id_{\g_1}-\left(P_1^\mu+\dots+P_N^\mu\right)$. Then the radical $\mathfrak r_\mu$ of $\omega_\mu$ is given by
\begin{align*}
\mathfrak r_\mu
& = \{x\in \g_1:\omega_\mu(x,x')=0\text{ for all }x'\in\g_1\}\\
& = \ker J_\mu  = \ker J_\mu^2 = \ran P_0^\mu.
\end{align*}
Hence we may choose $v_{0,1}^\mu,\dots,v_{0,r_0}^\mu\in \mathfrak r_\mu$ such that all the $v_{n,m}^\mu$, $n\in\{0,\dots,N\}$ form an orthonormal basis of $\g_1$. Note that the functions $\mu\mapsto v_{n,m}^\mu$ are just locally defined on neighborhoods of a fixed point $\mu_0\in\g_{2,r}^*$. These neighborhoods yield a covering of $\g_{2,r}^*$. Using this covering, we obtain a family $(U_\ell)_{\ell\in\N}$ disjoint Euclidean open subsets $U_\ell \subseteq\g_{2,r}^*$ whose union is $\g_{2,r}^*$ up to a set of measure zero, and we may define measurable functions $\mu\mapsto v_{n,m}^\mu$ on $\g_{2,r}^*$ that are real analytic on each $U_\ell$.

When decomposing $\R^{d_1}=\R^{r_0}\oplus\R^{2r_1}\oplus \dots\oplus \R^{2r_N}$ and sending the $(n,m)$-th standard basis vector of $\R^{d_1}$ onto $v_{n,m}^\mu$, we obtain a map
\[
R_\mu : \R^{r_0} \oplus \bigoplus_{n=1}^N \R^{2r_n} \to  \bigoplus_{n=0}^N {\ran P_n^\mu}
\]
such that
\[
P_n^\mu R_\mu=R_\mu P_n \quad \text{for all } n\in\{0,\dots,N\},
\]
where $P_n$ denotes the projection from $\R^{d_1}=\R^{r_0}\oplus\R^{2r_1}\oplus \dots\oplus \R^{2r_N}$ onto the $n$-th layer. If we let $\mathbf{b}^\mu=(b^\mu_1,\dots,b^\mu_N)$, we obtain
\begin{equation}\label{eq:transformation}
\big( L^\mu (\phi\circ R_\mu^{-1} )\big) \circ R_\mu
 = \atL{b^\mu}{r}{d_1} \phi
\end{equation}
for any $\phi\in\mathcal S(\R^{d_1})$, say. This can be seen, for instance, via direct computation, or, more conceptually, as follows: If $\delta$ denotes the Dirac measure at the identity element of the Lie group $G$, then, in distributional sense,
\[
Lf=L(f*\delta)=f*(L\delta), \quad f\in \mathcal S(G),
\]
where $*$ denotes the group convolution on $G$. Let $\Delta_{\g_1}$ denote the Euclidean Laplacian on $\g_1$. Since $X_j\delta=\partial_{x_j} \delta$ for each of the vector fields $X_j$, we get
\begin{equation}\label{eq:transformation-ii}
Lf=f*(-\Delta_{\g_1}\delta).
\end{equation}
Given $\phi,\psi\in\mathcal S(\g_1)$, we consider their \textit{$\mu$-twisted convolution} given by
\[
\phi \times_\mu \psi (x) = \int_{\g_1} \phi(x') \psi(x-x') e^{\frac i 2 \omega_\mu(x,x')} \,dx' , \quad x\in\g_1.
\]
Then, for $f,g\in \mathcal S(G)$, we have
\[
(f*g)^\mu = f^\mu \times_\mu g^\mu.
\]
Let $\delta_0$ be the Dirac measure at the origin of $\g_1$. Then \cref{eq:transformation-ii} yields
\[
L^\mu \phi = \phi \times_\mu (-\Delta_{\g_1}\delta_0).
\]
Thus, interpreted in distributional sense, we have
\begin{align*}
L^\mu (\phi \circ R_\mu^{-1} ) (R_\mu  y) 
& = \int_{\g_1} \phi(R_\mu^{-1}x') (-\Delta_{\g_1}\delta)(R_\mu y-x') e^{\frac i 2 \omega_\mu(R_\mu y, x')} \,dx' \\
& = \int_{\R^{d_1}} \phi(y') (-\Delta_{\g_1}\delta)(y-y') e^{\frac i 2\omega_\mu(R_\mu y,R_\mu y')} \,dy',
\end{align*}
which, in view of \cref{eq:symplectic}, gives \cref{eq:transformation}. 

Finally, since the rotation $R_\mu$ intertwines the operators $L^\mu$ and $\smash{\atL{b^\mu}{r}{d_1}}$ via conjugation in \cref{eq:transformation}, it also intertwines their functional calculi, that is,
\[
(F(L^\mu)\phi)\circ R_\mu = F\big(\atL{b^\mu}{r}{d_1}\big)(\phi\circ R_\mu)
\]
for all bounded Borel functions $F:\R\to\C$ and all $\phi\in L^2(\g_1)$, which is \cref{eq:rotation}.
\end{proof}

\begin{remark}
For our later purposes, we actually only need the function $\mu\mapsto R_\mu$ to be measurable without relying on the smoothness properties of (iii). This is due to the facts that the convolution kernels $\mathcal K_{F(L,\textbf{U})}$ of the operators $F(L,\textbf{U})$ are rotational invariant on each of the blocks given by the projections $P_n^\mu$, see \cref{prop:conv-kernel} and \cref{eq:Laguerre}, and that the proof of the $L^p$-$L^2$ restriction type estimate relies in particular on a Plancherel argument.
\end{remark}

\subsection{Spectral properties of anisotropic twisted Laplacians}

\hyphenation{an-iso-trop-ic}
Let again $d_1\in\N\setminus\{0\}$, $N\in\N\setminus\{0\}$, $\mathbf b=(b_1,\dots,b_N) \in [0,\infty)^N$, $\mathbf r=(r_1,\dots,r_N)\in(\N\setminus\{0\})^N$ and $r_0=d_1-2\left|\textbf{r}\right|_1$. Since the anisotropic twisted Laplacian
\[
\atL{b}{r}{d_1} =
(-\Delta_{\R^{r_0}}) \oplus \tL{b_1}{2r_1} \oplus \dots  \oplus \tL{b_N}{2r_N}
\]
acts as a Laplacian on the layer $\R^{r_0}$, we introduce a second partial Fourier transform
\begin{equation}\label{eq:partial-ft-0}
f^{(\tau)}( y) = \int_{\R^{r_0}} f(t, y) e^{- i \langle \tau, t\rangle } \, dt,\quad (\tau, y)\in \R^{r_0}\times\R^{d_1-r_0}.
\end{equation}
If we let $\bar d_1=d_1-r_0=2\left|\textbf{r}\right|_1$, then
\[
\big(\atL{b^\mu}{r}{d_1} f\big)^{(\tau)} = \big(|\tau|^2+\atL{b^\mu}{r}{\bar d_1}\big) f^{(\tau)}.
\]
Arguing similar to the proof of Proposition~1.1 of \cite{Mue90} (by comparing the generators of the corresponding semigroups of $\atL{b^\mu}{r}{d_1}$ and $|\tau|^2+\Delta^{(r,b^\mu)}$ on $\R^{\bar d_1}$), one can verify that the functional calculus of the two operators is compatible with the partial Fourier transform, whence
\begin{equation}\label{eq:Fourier-rad}
\big(F\big(\atL{b^\mu}{r}{d_1}\big) f\big)^{(\tau)} = F\big(|\tau|^2+\atL{b^\mu}{r}{\bar d_1}\big) f^{(\tau)}\quad\text{in } L^2(\R^{\bar d_1})
\end{equation}
for all bounded Borel functions $F:\R\to\C$, all $f\in L^2(\R^{d_1})$ and almost all $\tau\in\R^{d_0}$. Hence, by \cref{prop:rotation} and \cref{eq:Fourier-rad}, spectral properties of the twisted Laplacian $L^\mu$ are those of the operators $|\tau|^2+\atL{b^\mu}{r}{\bar d_1}$ modulo an orthogonal transformation and the partial Fourier transform \cref{eq:partial-ft-0}.

If the parameter $\textbf{b}=(b_1,\dots,b_N)$ satisfies $b_n>0$ for all $n\in\{1,\dots, N\}$, the spectral projections of the anisotropic twisted Laplacian $\atL{b}{r}{\bar d_1}$ can be written down in terms of twisted convolutions with Laguerre functions.

\begin{definition}\label{def:spectral-proj}
Let $\textbf{b}=(b_1,\dots,b_N)\in (0,\infty)^N$, $\textbf{r}=(r_1,\dots,r_N)\in (\N\setminus\{0\})^N$ with $N\in\N\setminus\{0\}$, and $\bar d_1=2r_1+\dots+2r_N$.

For $\phi,\psi\in L^2(\R^{\bar d_1})$, we call the function $\phi \times_{(\mathbf b,\mathbf r)} \psi$ given by
\[
\phi \times_{(\mathbf b,\mathbf r)} \psi(y) =
\int_{\R^{\bar d_1}} \phi(z) \psi(y - z) E^{\mathbf b,\mathbf r}(y,z )\,dz, \quad y\in \R^{\bar d_1}
\]
the \textit{$(\mathbf b,\mathbf r)$-twisted convolution} of $\phi$ and $\psi$, where $E^{\mathbf b,\mathbf r}$ is given by
\[
E^{\mathbf b,\mathbf r}(y, z)
= \prod_{n=1}^N \exp\big(\tfrac i 2 \,b_n\, \omega_{\R^{2r_n}}\big(y^{(n)},z^{(n)}\big)\big),
\]
with $y=(y^{(1)},\dots,y^{(N)}),z=(z^{(1)},\dots,z^{(N)})\in \R^{2r_1}\times\dots\times\R^{2r_N}$.

We define the \textit{$(\mathbf{b},\mathbf{r})$-rescaled Laguerre functions} $\Lag{k}{b}{r}$ via
\[ 
\Lag{k}{b}{r} = \varphi_{k_1}^{(b_1,r_1)} \otimes \dots \otimes \varphi_{k_N}^{(b_N,r_N)},\quad \textbf{k}=(k_1,\dots,k_N)\in\N^N,
\]
where $\varphi_k^{(\lambda,m)}$ denotes the $\lambda$-rescaled Laguerre function given by
\begin{equation}\label{eq:Laguerre}
\varphi_k^{(\lambda,m)}(z) = \lambda^m L_k^{m-1}\big(\tfrac 1 2 \lambda |z|^2\big) \,e^{-\frac 1 4 \lambda |z|^2},\quad z\in\R^{2m},
\end{equation}
and $L_k^{m-1}$ is the $k$-th Laguerre polynomial of type $m-1$.
\end{definition}

\begin{proposition}\label{prop:spectral}
If $\mathbf{b}\in (0,\infty)^N$, the spectrum of $\atL{b}{r}{\bar d_1}$ on $L^2(\R^{\bar d_1})$ consists of the eigenvalues
\[
\eigv{k}{b}{r} = \sum_{n=1}^N \left(2 k_n +r_n\right)b_n,\quad \mathbf{k}=(k_1,\dots,k_N)\in \N^N,
\]
and the associated orthogonal projections $\proj{k}{b}{r}$ given by
\[
\proj{k}{b}{r} f =  f \times_{(\mathbf b,\mathbf r)} \Lag{k}{b}{r},\quad f\in L^2(\R^{\bar d_1})
\]
decompose $L^2(\R^{\bar d_1})$ into subspaces of eigenspaces of $\atL{b}{r}{\bar d_1}$.
\end{proposition}

\begin{remark}
In general, $\proj{k}{b}{r}$ is only a projection onto a subspace of an eigenspace as two eigenvalues $\smash{\eigv{k}{b}{r}}$ and $\smash{\eigv{k'}{b}{r}}$ might coincide for $\mathbf{k}\neq \mathbf{k'}$. The projection onto the corresponding eigenspace is given by
\[
f\mapsto \sum_{\mathbf{k'}\in\N^N:\eigv{k}{b}{r}=\eigv{k'}{b}{r}} \proj{k'}{b}{r} f.
\]

\end{remark}

\begin{proof}
We briefly recall the spectral properties of twisted Laplacians. We refer the reader to \cite{Th93} for further details. For $\lambda>0$, the Schrödinger representation $\pi_\lambda$ of the Heisenberg group $\mathbb H_m = \C^m \times \R$ on $L^2(\R^m)$ is given by
\[
\pi_\lambda(a, b, t) \varphi(\xi) = e^{i \lambda t} e^{i \lambda(a \xi+\frac{1}{2} ab)} \varphi(\xi+b),
\]
where $a,b,\xi\in\R^m,t\in\R$ and $\varphi\in L^2(\R^m)$. For $\nu\in\N^m$, let $\Phi_\nu^{\lambda}$ be the $\lambda$-rescaled Hermite function given by
\[
\Phi_\nu^\lambda(\xi) = \lambda^{m/4} \prod_{j=1}^m h_{\nu_j}(\lambda^{1/2} \xi_j),\quad \xi\in \R^m,
\]
where $h_{\nu_j}$ shall denote the $\nu_j$-th Hermite function on $\R$ (as, for instance, defined in \cite[Equation (1.1.18)]{Th93}). It is well know \cite[pp.\ 16]{Th93} that the matrix coefficients $\Phi_{\alpha,\beta}$, $\alpha,\beta\in\N^m$ given by
\begin{equation}\label{eq:special-hermite}
\Phi_{\alpha,\beta}^\lambda(z) = (2\pi)^{-m/2} \lambda^{m/2} \big(\pi_\lambda(z,0) \Phi_\alpha^\lambda, \Phi_{\beta}^\lambda\big),\quad z\in\R^{2m}
\end{equation}
form a complete orthonormal system of eigenfunctions of the $\lambda$-twisted Laplacian $\tL{\lambda}{2m}$, with 
\begin{equation}\label{eq:eigenfunction-1}
\Delta_{\R^{2m}}^{(\lambda)} \Phi_{\alpha,\beta}^\lambda = \left(2\left|\beta\right|_1+m\right) \lambda \,\Phi_{\alpha,\beta}^\lambda.
\end{equation}
Thus, for $k\in\N$, the orthogonal projection associated with the eigenvalue $\left(2k+m\right)\lambda$ is given by
\[
P_k^{(\lambda,m)} f = \sum_{\alpha\in \N^m} \sum_{|\beta|_1 =k}  \left(f,\Phi_{\alpha,\beta}^\lambda\right)_{L^2(\R^{2m})}\Phi_{\alpha,\beta}^\lambda.
\]
The projection $P_k^{(\lambda,m)}$ can be written in a more explicit form via the $\lambda$-twisted convolution given by
\[
f\times_{(\lambda,m)} g (z) = \int_{\R^{2m}} f(w)g(z-w) e^{\frac i 2 \lambda\, \omega_{\R^{2m}}(z,w)}\,dw,\quad z\in\R^{2m},
\]
where $\omega_{\R^{2m}}$ shall again denote the standard symplectic form of \cref{eq:standard-symplectic-1}. Since 
\[
\Phi_{\nu,\nu'}^\lambda(z) = \lambda^{m/2} \, \Phi^1_{\nu,\nu'}(\lambda^{1/2}z),
\]
the identities (1.3.41) and (1.3.42) of \cite[pp.\ 21]{Th93} imply
\begin{equation}\label{eq:eigenfunction-4}
\varphi_k^{(\lambda,m)}(z) = (2\pi)^{m/2} \lambda^{m/2} \sum_{|\nu|_1=k} \Phi^\lambda_{\nu,\nu}(z).
\end{equation}
Hence, by (2.1.5) of \cite[p.\ 30]{Th93}, $P_k^{(\lambda,m)}$ may be rewritten as
\begin{equation}\label{eq:proj-compact}
P_k^{(\lambda,m)} f = f \times_{(\lambda,m)} \varphi_k^{(\lambda,m)}.
\end{equation}
Writing down the eigenfunctions of the anisotropic twisted Laplacian $\atL{b}{r}{\bar d_1}$ is now immediate, as the functions
\[
\eigf{\nu}{\nu'}{b}{r} = \Phi_{\nu^{(1)},(\nu')^{(1)}}^{(b_1)}\otimes\dots\otimes \Phi_{\nu^{(N)},(\nu')^{(N)}}^{(b_N)}
\]
with $\nu=(\nu^{(1)},\dots,\nu^{(N)}),\nu'=\big((\nu')^{(1)},\dots,(\nu')^{(N)}\big)\in \N^{r_1}\times \dots\times\N^{r_N} =\N^{\bar d_1/2}$ form a complete orthonormal system of $\R^{\bar d_1}$, with
\[
\atL{b}{r}{d_1} \, \eigf{\nu}{\nu'}{b}{r}
= \bigg( \sum_{n=1}^N \big(2\, |(\nu')^{(n)}|_1+r_n\big)b_n \bigg) \eigf{\nu}{\nu'}{b}{r}.
\]
Hence $L^2(\R^{\bar d_1})$ decomposes into eigenspaces of $\atL{b}{r}{\bar d_1}$. Using \cref{eq:proj-compact} on every block shows that the projections defined by
\[
\proj{k}{b}{r} f =  f \times_{(\mathbf b,\mathbf r)} \Lag{k}{b}{r}
\]
satisfy
\[
\proj{k}{b}{r} f = \sum_{\nu\in \N^{\bar d_1/2}} \sum_{\nu'\in A_{\textbf{k}}} \big(f,\eigf{\nu}{\nu'}{b}{r}\big)_{L^2(\R^{d_1})}\eigf{\nu}{\nu'}{b}{r},
\]
where $A_{\textbf{k}}$ is the set of all $\nu'\in \N^{\bar d_1/2}$ such that $|(\nu')^{(n)}|_1 =k_n$ for all $1\le n\le N$. Hence $\proj{k}{b}{r}$ projects onto a subspace of the eigenspace of $\atL{b}{r}{\bar d_1}$ that is associated with the eigenvalue
\[
\eigv{k}{b}{r} = \sum_{n=1}^N \left(2 k_n +r_n\right)b_n,\quad \textbf{k}\in \N^N,
\]
which finishes the proof.
\end{proof}

\subsection{Convolution kernels}

Recall that the sub-Laplacian
\[
L=-(X_1^2+\dots+X_{d_1}^2)
\]
and the vector $\textbf{U}=(-iU_1,\dots,-iU_{d_2})$ of differential operators, where $U_1,\dots,U_{d_2}$ is the chosen basis of the second layer of the stratification $\g=\g_1\oplus \g_2$, admit a joint functional calculus which is compatible with the partial Fourier transform along $\g_2^*$ by \cref{prop:joint-calculus}. For suitable functions $F:\R\times \R^{d_2}\to\C$, the operator $F(L,\mathbf U)$ possesses a convolution kernel $\mathcal K_{F(L,\mathbf U)}$, that is,
\[
F(L,\mathbf U) f = f * \mathcal K_{F(L,\mathbf U)} \quad \text{for all } f\in \S(G).
\]
As in \cite[Corollary 8]{MaMue13}, we show that the convolution kernel $\mathcal K_{F(L,\mathbf U)}$ can be explicitly written down in terms of the Fourier transform and rescaled Laguerre functions. To that end, recall that \cref{prop:rotation} yields the spectral decomposition
\[
-J_\mu^2 = \sum_{n=1}^N (b_n^\mu)^2 P_n^\mu
\]
for all $\mu$ in the Zariski open subset $\g_{2,r}^*\subseteq\g_2^*$ of \cref{prop:rotation}. We use the notation of \cref{prop:rotation} in the next proposition.

\begin{proposition}\label{prop:conv-kernel}
If $F:\R\times \R^{d_2}\to\C$ is a Schwartz function, then $F(L,\mathbf U)$ possesses a convolution kernel $\mathcal K_{F(L,\mathbf U)}\in \S(G)$. For $x\in\g_1$ and $u\in\g_2^*$, we have
\begin{align}\label{eq:conv-kernel}
\mathcal K_{F(L,\mathbf U)}(x,u) = & (2\pi)^{-r_0-d_2} \int_{\g_{2,r}^*}   \int_{\R^{r_0}} \sum_{\mathbf{k}\in\N^N} F(|\tau|^2+\eigvp{k}{\mu},\mu) \notag \\
& \times \bigg[\prod_{n=1}^N \varphi_{k_n}^{(b_n^\mu,r_n)}(R_\mu^{-1} P_n^\mu x)\bigg] e^{i\langle \tau, R_\mu^{-1} P_0^\mu x\rangle} \, e^{i\langle \mu, u\rangle} \, d\tau  \, d\mu,
\end{align}
where $P_0^\mu = \id_{\g_1} - \left(P_1^\mu+\dots+P_N^\mu\right)$ and
\[
\eigvp{k}{\mu} = \eigv{k}{b^\mu}{r} = \sum_{n=1}^N \left(2k_n+r_n\right)b_n^\mu.
\]
\end{proposition}

\begin{proof}
Since $F$ is a Schwartz function, we have $\mathcal K_{F(L,\mathbf U)}\in \S(G)$ due to a result of Hulanicki \cite{Hu84} (see also \cite[Proposition 4.2.1]{Ma10}). The formula for $\mathcal K_{F(L,\mathbf U)}$ can be proved as in \cite[Proposition 4]{Ma15} by using the Fourier inversion formula of the group Fourier transform on $G$ and computing the matrix coefficients of the Schrödinger representations, but can also be derived as follows. Recall from \cref{prop:rotation} that
\[
P_n^\mu R_\mu=R_\mu P_n\quad \text{for all }\mu \in \g_{2,r}^* \text{ and } n\in\{0,\dots,N\},
\]
where $P_n$ denotes the projection from $\R^{d_1}=\R^{r_0}\oplus\R^{2r_1}\oplus \dots\oplus \R^{2r_N}$ onto the $n$-th layer. Let $\bar P = \mathrm{id}_{\R^{d_1}} - P_0$ and $\bar P^\mu=\id_{\g_1} -P_0^\mu$. Then
\[
\bar P^\mu R_\mu=R_\mu \bar P \quad \text{for all }\mu \in \g_{2,r}^*.
\]
Using the Fourier inversion formula, \cref{prop:joint-calculus,prop:rotation}, and \cref{eq:Fourier-rad}, we obtain
\begin{align}
& (2\pi)^{r_0+d_2} F(L,\textbf{U}) f (x,u)
 = (2\pi)^{r_0} \int_{\g_2^*} F(L^\mu,\mu) f^\mu (x) e^{i\langle\mu,u\rangle} \,d\mu \notag \\
& = (2\pi)^{r_0} \int_{\g_{2,r}^*} \big( F( \atL{b^\mu}{r}{d_1} ,\mu) (f^\mu \circ R_\mu)\big) (R_\mu^{-1}x) e^{i\langle\mu,u\rangle} \,d\mu \notag\\
& =  \int_{\g_{2,r}^*} \int_{\R^{r_0}} \big( F( |\tau|^2 +  \atL{b^\mu}{r}{\bar d_1} ,\mu) f^{(\tau,\mu)} \big)(\bar P R_\mu^{-1}x) \, e^{i\langle\tau,P_0 R_\mu^{-1}x \rangle}e^{i\langle\mu,u\rangle} \,d\tau \,d\mu,\notag \\
& = \int_{\g_{2,r}^*} \int_{\R^{r_0}}  \sum_{\mathbf{k}\in\N^N}  F( |\tau|^2 + \eigvp{k}{\mu} ,\mu) \big(\proj{k}{b^\mu}{r} f^{(\tau,\mu)}\big)  (\bar P R_\mu^{-1}x)\notag \\
& \hspace*{6cm} \times e^{i\langle\tau,P_0 R_\mu^{-1} x\rangle}e^{i\langle\mu,u\rangle} \,d\tau \,d\mu,\label{eq:conv-kernel-a}
\end{align}
where we put
\[
f^{(\tau,\mu)} = (f^\mu \circ R_\mu)^{(\tau)}.
\]
By \cref{prop:spectral}, we have
\begin{align}
\big(\proj{k}{b^\mu}{r} f^{(\tau,\mu)}\big)(y) = \int_{\R^{\bar d_1}} f^{(\tau,\mu)}(z)\, \Lag{k}{b^\mu}{r}(y - z)\, E^{\mathbf b,\mathbf r}(y,z )\,dz, \label{eq:conv-kernel-b}
\end{align}
where $E^{\mathbf b,\mathbf r}$ is given by
\[
E^{\mathbf b,\mathbf r}(y, z)
= \prod_{n=1}^N \exp\big(\tfrac i 2 \,b_n\, \omega_{\R^{2r_n}}\big(y^{(n)},z^{(n)}\big)\big),
\]
with $y=(y^{(1)},\dots,y^{(N)}),z=(z^{(1)},\dots,z^{(N)})\in \R^{2r_1}\times\dots\times\R^{2r_N}$.

Unboxing the definition of $f^{(\tau,\mu)}$ yields
\begin{equation}\label{eq:conv-kernel-c}
f^{(\tau,\mu)}(z) 
 = \int_{\R^{r_0}} \int_{\g_2} f\big(R_\mu(t,z),u'\big)\, e^{-i\langle \tau,t\rangle} e^{-i\langle \mu,u'\rangle}\,du' dt.
\end{equation}
Note that
\begin{align*}
f * \mathcal K_{F(L,\mathbf U)} (x,u)
& = \int_G f(x',u') \mathcal K_{F(L,\mathbf U)}\big((x',u')^{-1}(x,u)\big)\,d(x,u) \\
& = \int_G f(x',u') \mathcal K_{F(L,\mathbf U)}\big(x-x',u-u'+\tfrac 1 2 [x,x']\big)\,d(x,u).
\end{align*}
Inserting \cref{eq:conv-kernel-b} and \cref{eq:conv-kernel-c} into the last line of \cref{eq:conv-kernel-a} and rearranging the order of integration yields the expression
\begin{align*}
 \int_{\R^{\bar d_1}} \int_{\R^{r_0}} \int_{\g_2} \int_{\g_{2,r}^*} \int_{\R^{r_0}} \sum_{\mathbf{k}\in\N^N}  & f\big(R_\mu(t,z),u'\big) \,F( |\tau|^2 + \eigvp{k}{\mu} ,\mu) \,   \,\\
& \times \Lag{k}{b^\mu}{r}(\bar P R_\mu^{-1}x - z)\, E^{\mathbf b,\mathbf r}(\bar P R_\mu^{-1}x,z )\\
& \times e^{i\langle\tau,P_0 R_\mu^{-1} x-t\rangle}e^{i\langle\mu,u-u'\rangle} \,d\tau \,d\mu\,du' dt\,dz.
\end{align*}
Recall that \cref{eq:symplectic-0} asserts
\[
\omega_\mu(P_n^\mu x,P_n^\mu x')=b_n^\mu \, \omega_{\R^{2r_n}}(P_n R_\mu^{-1}x,P_n R_\mu^{-1}x').
\]
Since $\omega_\mu(x,x')=\langle J_\mu x,x'\rangle$ and
\[
J_\mu(\ran P_n^\mu)\subseteq \ran P_n^\mu
\]
by \cref{prop:rotation}, we have
\[
\omega_\mu(P_n^\mu x,P_{n'}^\mu x')=0 \quad\text{for } n\neq n'.
\]
This yields
\begin{align*}
E^{\mathbf b,\mathbf r}(\bar P R_\mu^{-1}x,\bar P R_\mu^{-1}x')
& = \exp\big( \tfrac i 2 \omega_\mu(\bar P^\mu x, \bar P^\mu x') \big) \\
& = \exp\big( \tfrac i 2 \omega_\mu(x,x') \big).
\end{align*}
Hence, substituting $x' = R_\mu(t,z)$ in the formula above, that is,
\[
t = P_0 R_\mu^{-1}x'\quad\text{ and } \quad z = \bar P R_\mu^{-1} x',
\]
and using $\bar P^\mu R_\mu=R_\mu \bar P$ and $P_0^\mu R_\mu=R_\mu P_0$, we obtain \cref{eq:conv-kernel}.
\end{proof}

\section{Spectral cluster estimates} \label{sec:cluster}

Throughout this section, we fix $d_1\in\N\setminus\{0\}$ and $\mathbf r=(r_1,\dots,r_N)\in(\N\setminus\{0\})^N$ with $N\in\N\setminus\{0\}$ and consider the anisotropic twisted Laplacians
\[
\atL{b}{r}{d_1} =
(-\Delta_{\R^{r_0}}) \oplus \tL{b_1}{2r_1} \oplus \dots  \oplus \tL{b_N}{2r_N}
\]
for parameters $\textbf{b}=(b_1,\dots,b_N)$ ranging in a compact subset $\textbf{B}$ of $[0,\infty)^N$. In the following, we deviate slightly from the notation of the previous section (where we generally denoted coordinates on $\g_1$ by $x$ and coordinates on $\R^{d_1}$ under the orthogonal transformation $R_\mu:\R^{d_1}\to\g_1$ of \cref{prop:rotation} by $y$) and will denote the coordinates on $\R^{d_1}$ again by $x$ in this section.

Given $n\in\N\setminus\{0\}$, we again write $p_n=2(n+1)/(n+3)$ for the Stein--Tomas threshold. The main result of this section are the following spectral cluster estimates, which are inspired by \cite[Chapter 5]{So93}.

\begin{theorem}[Spectral cluster estimates]\label{thm:cluster}
Let $\mathbf{B}\subseteq [0,\infty)^N$ be a compact subset. If $1\le p \le p_{d_1}$, then
\begin{equation}\label{eq:cluster}
\norm{\mathbf{1}_{[K,K+1)}(\atL{b}{r}{d_1}) }_{p\to 2} \le C_{\mathbf B,\mathbf r,p} \left(K+1\right)^{\frac{d_1}2(\frac 1 p - \frac 1 2)-\frac 1 2}
\end{equation}
for all $K\in \N$ and $\mathbf b\in \mathbf B$.
\end{theorem}

We give two proofs of these estimates. The first one relies on a Mehler type formula and subordination by the heat semigroup associated with the anisotropic twisted Laplacian, but only works for the smaller range $1\le p <p_{d_1-1}$.

\subsection{First proof via subordination by the heat semigroup}

\begin{proposition}\label{prop:cluster-v1}
The cluster estimate \eqref{eq:cluster} holds true for $1\le p <p_{d_1-1}$.
\end{proposition}

The proof of \cref{prop:cluster-v1} is essentially that of Proposition II.8 of \cite{ChOuSiYa16} combined with a Mehler type formula. However, the uniformity of the constants is crucial here, so we give full details.

\begin{proof}
Given $\textbf{b}\in [0,\infty)^N$, the heat semigroup
\[
T^{\textbf{b},\textbf{r}}(t)=\exp(-t\atL{b}{r}{d_1}),\quad t\ge 0
\]
admits an analytic extension $\zeta\mapsto T^{\textbf{b},\textbf{r}}(\zeta)=\exp(-\zeta\atL{b}{r}{d_1})$ to the complex right half-plane. By \cite[Proposition~4]{MaMue16}, it can be written as a twisted convolution with the heat kernel $p^{\mathbf b,\mathbf r}_\zeta$, which is given by
\begin{align}
p^{\mathbf b,\mathbf r}_\zeta (x) = \frac{1}{(4\pi \zeta)^{d_1/2}} & \exp\Big(-\frac{1}{4\zeta} \big| x^{(0)}\big|^2\Big) \notag \\
& \times \prod_{n=1}^N \mathrm S(i\zeta b_n)^{r_n} \exp\Big(-\frac{1}{4\zeta}\, \mathrm T(i\zeta b_n) \big|x^{(n)}\big|^2\Big),\label{eq:heat-kernel}
\end{align}
where $x=(x^{(0)},\dots,x^{(N)}),x'=((x')^{(0)},\dots,(x')^{(N)})\in \R^{r_0}\oplus\R^{2r_1}\oplus \dots\oplus \R^{2r_N}$, and
\[
\mathrm S,\mathrm T: \C \setminus \{  k \pi : 0 \neq k \in \Z \} \to\C
\]
are the meromorphic functions given by
\[
\mathrm S(\zeta) = \frac{\zeta}{\sin \zeta}\quad\text{and}\quad 
\mathrm T(\zeta) = \frac{\zeta}{\tan \zeta}.
\]
Then, if
\[
E^{\mathbf b,\mathbf r}(x,x')
= \prod_{n=1}^N \exp\big(\tfrac i 2 \,b_n\, \omega_{\R^{2r_n}}\big(x^{(n)},(x')^{(n)}\big)\big),
\]
the operator $T^{\textbf{b},\textbf{r}}(\zeta)$ is given by
\begin{equation}\label{eq:heat-twisted}
T^{\textbf{b},\textbf{r}}(\zeta) f (x) = \int_{\R^{d_1}} f(x')p^{\mathbf b,\mathbf r}_\zeta(x-x') E^{(\textbf{b},\textbf{r})}(x,x') \,dx'.
\end{equation}
Note that the parameters $b_n$ are actually allowed to be zero, in which case the twisted convolution is just the Euclidean convolution on the $n$-th block. Note that
\begin{align*}
\frac 1 \zeta \, \mathrm T(i\zeta)
& = \coth\zeta
  = \frac{\sinh(2\Re\zeta)-i\sin(2\Im\zeta)}{2\left|\sinh\zeta\right|^2},
\end{align*}
whence
\[
\Re\Big( \frac 1 \zeta\, \mathrm T(i\zeta)\Big) \ge 0 \quad\text{if } \Re\zeta \ge 0. 
\]
Thus, neglecting all oscillations of the twisted convolution above, \cref{eq:heat-twisted} yields
\begin{equation}\label{eq:heat-dispersive}
\norm{T^{\textbf{b},\textbf{r}}(\zeta) f}_\infty
 \le \norm{f}_1 \norm{p^{\mathbf b,\mathbf r}_\zeta}_\infty
\lesssim_{\textbf{B},\mathbf r} \,|\zeta|^{-d_1/2} \norm{f}_1 
\end{equation}
for all $\zeta\in \C$ lying in the rectangle
\[
R = \{ \zeta \in\C : 0<\Re\zeta\le 1 \text{ and } |\Im \zeta|\le \alpha_{\textbf{B},\textbf{r}} \},
\]
where the constant $\alpha_{\textbf{B},\textbf{r}}>0$ is chosen small enough such that $\alpha_{\textbf{B},\textbf{r}}<\pi / b_n$ for all $\textbf{b}\in\textbf{B}$ and all $n\in\{1,\dots,N\}$. Interpolating between \cref{eq:heat-dispersive} and the trivial $L^2$-$L^2$ estimate, we get
\begin{equation}\label{eq:heat-interpol}
\norm{T^{\textbf{b},\textbf{r}}(\zeta)}_{p\to p'} \lesssim_{\mathbf B,\mathbf r} |\zeta|^{-\frac {d_1} 2(\frac 1 p - \frac 1 {p'})} \quad \text{for all } \zeta\in R \text{ and } 1\le p\le 2.
\end{equation}
Given $K\in\N$, we define
\[
\widehat{F_K}(\xi) = (1-|\xi|)_+^2\, e^{-iK\xi},\quad\xi\in\R.
\]
Then $F_K$ is given by
\[
F_K(\lambda) = \frac{2}{\pi}\, \frac{\lambda-K-\sin(\lambda-K)}{(\lambda-K)^3},\quad \lambda\in\R.
\]
In particular, $F_K:\R\to\C$ is a positive function with $\supp \widehat{F_K}=[-1,1]$. Put
\[
F_K^{\textbf{B},\textbf{r}}(\lambda) = F_K\big(\tfrac 1 2\alpha_{\textbf{B},\textbf{r}}\lambda\big)
\quad\text{and}\quad
G_K^{\textbf{B},\textbf{r}}(\lambda) = F_K^{\textbf{B},\textbf{r}}(\lambda)\, e^{-\tfrac 1 2\alpha_{\textbf{B},\textbf{r}}\lambda/(K+1)}.
\]
Note that there is some constant $C_{\textbf{B},\textbf{r}}>0$ independent of $K$ such that
\[
\inf_{\lambda \in [K,K+1)} G_K^{\textbf{B},\textbf{r}}(\lambda) \ge C_{\textbf{B},\textbf{r}}.
\]
This implies
\[
C_{\textbf{B},\textbf{r}} \,\norm{\mathbf{1}_{[K,K+1)}(\atL{b}{r}{d_1})f}_2 \le \norm{G_K^{\textbf{B},\textbf{r}}(\atL{b}{r}{d_1})f}_2.
\]
Using a $TT^*$ argument, we obtain
\begin{equation}\label{eq:T^*T}
C_{\textbf{B},\textbf{r}}^2\, \big\Vert\mathbf{1}_{[K,K+1)}(\atL{b}{r}{d_1})\big\Vert_{p\to p'} \le \big\Vert(G_K^{\textbf{B},\textbf{r}})^2(\atL{b}{r}{d_1})\big\Vert_{p\to p'}.
\end{equation}
Note that
\[
\supp \widehat{F_K^{\textbf{B},\textbf{r}}} \subseteq [-\tfrac 1 2\alpha_{\textbf{B},\textbf{r}},\tfrac 1 2 \alpha_{\textbf{B},\textbf{r}}].
\]
The Fourier inversion formula yields
\begin{equation}
(G_K^{\textbf{B},\textbf{r}})^2(\lambda) 
 = \frac{1}{(2\pi)^2} \int_{-\alpha_{\textbf{B},\textbf{r}}}^{\alpha_{\textbf{B},\textbf{r}}} \widehat{F_K^{\textbf{B},\textbf{r}}}*\widehat{F_K^{\textbf{B},\textbf{r}}} (\xi)\, e^{-(\alpha_{\textbf{B},\textbf{r}}/(K+1)-i\xi)\lambda} \,d\xi.
\end{equation}
Moreover, note that there is a constant $\tilde C_{\textbf{B},\textbf{r}}>0$ independent of $K$ such that
\begin{equation}\label{eq:conv-K}
\big\| \widehat{F_K^{\textbf{B},\textbf{r}}}*\widehat{F_K^{\textbf{B},\textbf{r}}}\big\|_\infty\le \tilde C_{\textbf{B},\textbf{r}}.
\end{equation}%
Altogether, \cref{eq:T^*T}, \cref{eq:conv-K} and \cref{eq:heat-interpol} yield 
\begin{align*}
& \big\Vert\mathbf{1}_{[K,K+1)}(\atL{b}{r}{d_1})\big\Vert_{p\to p'}
\lesssim_{\mathbf B,\mathbf r} \big\Vert(G_K^{\textbf{B},\textbf{r}})^2(\atL{b}{r}{d_1})\big\Vert_{p\to p'}\\
& \lesssim \int_{-\alpha_{\textbf{B},\textbf{r}}}^{\alpha_{\textbf{B},\textbf{r}}} \big|\widehat{F_K^{\textbf{B},\textbf{r}}}*\widehat{F_K^{\textbf{B},\textbf{r}}}(\xi)\big| \, \big\Vert e^{-(\alpha_{\textbf{B},\textbf{r}}/(K+1)-i\xi) \atL{b}{r}{d_1}}\big\Vert_{p\to p'}\,d\xi\\
& \lesssim_{\mathbf B,\mathbf r} \int_{-\infty}^\infty (\alpha_{\textbf{B},\textbf{r}}^2(K+1)^{-2}+\xi^2)^{-\frac {d_1} 4(\frac 1 p - \frac 1 {p'})}\,d\xi\\
& \sim_{\mathbf B,\mathbf r} (K+1)^{\frac {d_1} 2(\frac 1 p - \frac 1 {p'})-1}.
\end{align*}
Note that the last integral converges since $1\le p<p_{d_1-1}$. Hence
\[
\big\Vert\mathbf{1}_{[K,K+1)}(\atL{b}{r}{d_1})\big\Vert_{p\to 2} \le C_{\textbf{B},\mathbf r} \left(K+1\right)^{\frac {d_1} 2(\frac 1 p - \frac 1 2)-\frac 1 2},
\]
which is \eqref{eq:cluster}.
\end{proof}

\begin{remark}
In fact, formulas for heat kernels on two-step stratified Lie groups of the form \cref{eq:heat-kernel} are well-known, see for instance \cite{Hu76}, \cite[Corollary (5.5)]{Cy79}, \cite[Theorem 5.2]{MueRi03}, and \cite{MaMue16} for further references. Alternatively, \cref{eq:heat-kernel} could also be verified directly from \cref{prop:spectral} by applying the Mehler type formula for Laguerre functions of \cite[p.~37]{Th93} on each block of $\R^{r_0}\oplus\R^{2r_1}\oplus\dots\oplus\R^{2r_N}$. 
\end{remark}

\subsection{Second proof via the dispersive estimates of Koch and Tataru}

The proof of \cref{eq:cluster} for the full range $1\le p\le p_{d_1}$ relies on interpolation between $p=1$ and the endpoint $p=p_{d_1}$. For $p=1$, we just resort to our previous result. Let $\mathbf{B}\subseteq [0,\infty)^N$ be the compact subset of \cref{thm:cluster}.

\begin{lemma}\label{lem:cluster-p=1}
The cluster estimate \eqref{eq:cluster} holds true for $p=1$, i.e.,
\begin{equation}\label{eq:cluster-p=1}
\big\Vert\mathbf{1}_{[K,K+1)}(\atL{b}{r}{d_1})\big\Vert_{1\to 2} \le C_{\mathbf B,\mathbf r} \left(K+1\right)^{\frac {d_1} 4 -\frac 1 2}
\end{equation}
for all $K\in \N$ and $\mathbf b\in \mathbf B$.
\end{lemma}

\begin{proof}
Note that $d_1=r_0+2\left|\textbf{r}\right|_1 \ge 2$ and $N\ge 1$. We distinguish the cases $d_1=2$ and $d_1>2$. If $d_1>2$, then $p_{d_1-1}=2d_1/(d_1+2)>1$, and the statement follows directly from \cref{prop:cluster-v1}. If $d_1=2$, we have $r_0=0$, $N=1$, $\textbf{r}=1$, and $\textbf{b}=b_1$, so the anisotropic twisted Laplacian admits only one block
\[
\atL{b}{r}{d_1} = \tL{b_1}{2}.
\]
By \cref{eq:eigenfunction-1}, the spectrum of $\tL{b_1}{2}$ consists of the eigenvalues $\left(2k+1\right)b_1$, $k\in \N$. Then, similar to Lemma 3.1 in \cite{Ni23}, \cref{eq:cluster-p=1} follows from \cite{KoRi07}.
\end{proof}

For $p=p_{d_1}$, the exponent on the right-hand side of \eqref{eq:cluster} is given by
\begin{align*}
\frac{d_1}{2} \Big( \frac{1}{p_{d_1}} - \frac{1}{2} \Big) - \frac{1}{2}
& = \frac{d_1}{2} \Big( \frac{d_1+3}{2(d_1+1)} - \frac{1}{2} \Big) - \frac{1}{2} \\
& = - \frac{1}{2(d_1+1)}.
\end{align*}
Hence, via interpolation with $p=1$ from \cref{lem:cluster-p=1}, it suffices to prove
\begin{equation}\label{eq:endpoint}
\big\Vert\mathbf{1}_{[K,K+1)}(\atL{b}{r}{d_1})\big\Vert_{p_{d_1}\to 2} \le C_{\mathbf B,\mathbf r} \left(K+1\right)^{-1/(2(d_1+1))}.
\end{equation}

Similar to \cite[Lemma~2]{KoRi07}, we first reduce \cref{eq:endpoint} to some localized form. Let $B_1,B_2\subseteq\R^{d_1}$ denote the open balls of radius 1 and 2 centered at the origin. 

\begin{lemma}\label{lem:reduction}
There is a constant $\gamma=\gamma_{\mathbf B,\mathbf r}>0$ such that
\begin{equation}\label{eq:rescaling-3}
\kappa^{2/(d_1+1)} \| u|_{B_1} \|_{p_{d_1}'} \le C_{\mathbf B,\mathbf r}\, \kappa \, \norm*{ u|_{B_2}}_2 + \kappa^{-1} \norm*{ f|_{B_2}}_2
\end{equation}
for all $\mathbf{b}\in\mathbf{B}$, all $\kappa\ge 1$ and all $u\in L^2(\R^{d_1})$ and $f\in L^2(\R^{d_1})$ such that
\begin{equation}\label{eq:PDE-reduced}
\big(\Delta_{\R^{d_1}}^{\gamma^2\kappa^2 \mathbf{b},\mathbf{r}} - \gamma^2 \kappa^4 \big)u = f.
\end{equation}
\end{lemma}

Before we prove \cref{lem:reduction}, we show how the endpoint cluster estimate \cref{eq:endpoint} (and thus \cref{thm:cluster}) can be derived from it. We will apply \cref{lem:reduction} with $\kappa=(K+1)^{1/2}$. The parameter $\gamma=\gamma_{\textbf{B},\textbf{r}}>0$ exists only for technical reasons and will be chosen later to be sufficiently small, depending on the set $\textbf{B}$.

\begin{proof}[Proof of \cref{eq:endpoint}]
Given $K\in\N$, let $\cluster{K}{b}{r}$ denote the projection given by
\[
\cluster{K}{b}{r}=\mathbf{1}_{[K,K+1)}(\atL{b}{r}{d_1}).
\]
We reduce \cref{eq:endpoint} to \cref{eq:rescaling-3}. Let $\kappa=(K+1)^{1/2}$. Suppose we have the estimate
\begin{equation}\label{eq:reduction-PDE-1}
\kappa^{1/(d_1+1)} \norm{u}_{p_{d_1}'} \lesssim_{\mathbf B,\mathbf r} \norm{u}_2 + \norm{f}_2
\end{equation}
for all $\textbf{b}\in\textbf{B}$ and all $u\in L^2(\R^{d_1})$ and $f\in L^2(\R^{d_1})$ such that
\begin{equation}\label{eq:reduction-PDE-2}
(\atL{b}{r}{d_1} - \kappa^2) u = f.
\end{equation}
Then, given $u_0\in L^2(\R^{d_1})$, we can put
\[
u:=\cluster{K}{b}{r} u_0\quad\text{and} \quad f:=(\atL{b}{r}{d_1} - \kappa^2) \cluster{K}{b}{r} u_0,
\]
and \eqref{eq:reduction-PDE-1} would yield
\begin{align*}
\kappa^{1/(d_1+1)} \norm{\cluster{K}{b}{r} u_0}_{p_{d_1}'}
& \lesssim_{\mathbf B,\mathbf r} \norm{\cluster{K}{b}{r} u_0}_2 + \norm{(\atL{b}{r}{d_1} - \kappa^2) \cluster{K}{b}{r} u_0}_2 \\
& \le 2 \left\|u_0\right\|_2,
\end{align*}
which implies \eqref{eq:endpoint} by duality. So it remains to reduce \eqref{eq:reduction-PDE-1} to \cref{eq:rescaling-3}.

Let $\tilde \kappa=\gamma\kappa$. We cover the space with Euclidean balls $B_{\tilde\kappa}(x_j)$ with centers $x_j\in \R^{d_1}$ such that the dilated balls $B_{2\tilde\kappa}(x_j)$ have only bounded overlap. Then, in place of \cref{eq:reduction-PDE-1}, it suffices to show
\begin{equation}\label{eq:localization-1}
\kappa^{1/(d_1+1)} \norm{u|_{B_{\tilde\kappa}(x_j)}}_{p_{d_1}'} \lesssim_{\mathbf B,\mathbf r} \norm{u|_{B_{2\tilde\kappa}(x_j)}}_2 + \norm{f|_{B_{2\tilde\kappa(x_j)}}}_2.
\end{equation}
Indeed, using \eqref{eq:localization-1} in combination with Minkowski's inequality and the bounded overlap of the balls $B_{2\tilde\kappa}(x_j)$, we get
\begin{align*}
\kappa^{1/(d_1+1)} \norm*{u}_{p_{d_1}'}
 & \le \kappa^{1/(d_1+1)} \bigg( \sum_j \big\Vert u|_{B_{\tilde\kappa}(x_j)}\big\Vert_{p_{d_1}'}^{p_{d_1}'} \bigg)^{1/p_{d_1}'}\\
 & \lesssim_{\mathbf B,\mathbf r} \Bigg( \sum_j \left(\big\Vert u|_{B_{2\tilde\kappa}(x_j)} \big\Vert_2 + \big\Vert f|_{B_{2\tilde\kappa(x_j)}}\big\Vert_2 \right)^{p_{d_1}'} \Bigg)^{1/p_{d_1}'}\\
 &\lesssim_{\mathbf B,\mathbf r} \norm*{u}_2 + \norm*{f}_2.
\end{align*}
For $v\in L^2(\R^{d_2})$, we consider the twisted convolution given by
\[
v \times_{(\mathbf b,\mathbf r)} u (x) =
\int_{\R^{d_1}} v(x') u(x - x') E^{\mathbf b,\mathbf r}(x,x')\,dz, \quad x'\in \R^{\bar d_1}
\]
where $E^{\mathbf b,\mathbf r}$ is given by
\[
E^{\mathbf b,\mathbf r}(x, x')
= \prod_{n=1}^N \exp\big(\tfrac i 2 \,b_n\, \omega_{\R^{2r_n}}\big(y^{(n)},z^{(n)}\big)\big),
\]
with $x=(x^{(0)},\dots,x^{(N)}),x'=((x')^{(0)},\dots,(x')^{(N)})\in \R^{r_0}\oplus\R^{2r_1}\oplus \dots\oplus \R^{2r_N}$, and where $\omega_{\R^{2r_n}}$ is the standard symplectic form of \cref{eq:standard-symplectic-1}. Then
\[
\atL{b}{r}{d_1} (v \times_{(\textbf{b},\textbf{r})} u ) = v \times_{(\textbf{b},\textbf{r})} (\atL{b}{r}{d_1} u)\quad \text{for all } v\in L^2(\R^{d_1}).
\]
This implies
\[
\atL{b}{r}{d_1} \big(u(x-x') E^{(\textbf{b},\textbf{r})}(x,x')\big) = \big(\atL{b}{r}{d_1} u\big)(x-x') E^{(\textbf{b},\textbf{r})}(x,x').
\]
As a consequence, substituting
\[
u(x)=\tilde u(x-x') E^{(\textbf{b},\textbf{r})}(x,x')
\quad\text{and}\quad
f(x)=\tilde f(x-x') E^{(\textbf{b},\textbf{r})}(x,x')
\]
shows that we may assume without loss of generality $x_j=0$. Hence, in place of \cref{eq:localization-1}, it suffices to show
\begin{equation}\label{eq:localization-2}
\kappa^{1/(d_1+1)} \norm*{u|_{B_{\tilde\kappa}}}_{p_{d_1}'} \lesssim_{\mathbf B,\mathbf r} \norm*{u|_{B_{2\tilde\kappa}}}_2 + \norm*{f|_{B_{2\tilde\kappa}}}_2
\end{equation}
for all $u$ satisfying \eqref{eq:reduction-PDE-2} (where we just write $B_\kappa$ and $B_{2\kappa}$ for the centered balls $B_\kappa(0)$ and $B_{2\kappa}(0)$). On the other hand, rescaling with $\tilde \kappa=\gamma\kappa$ shows that \eqref{eq:localization-2} is equivalent to
\begin{align}
\kappa^{1/(d_1+1)+d_1/p_{d_1}'} \norm*{ u|_{B_{\tilde\kappa}}(\tilde\kappa\,\cdot\,) }_{p_{d_1}'}
 \lesssim_{\mathbf B,\mathbf r} \ & \kappa^{d_1/2} \left\|u|_{B_{2\tilde\kappa}}(\tilde\kappa\,\cdot\,)\right\|_2 \notag  \\ & + \kappa^{d_1/2} \left\| f|_{B_{2\tilde\kappa}}(\tilde\kappa\,\cdot\,)\right\|_2.\label{eq:rescaling-1}
\end{align}
Since $\Delta_{\R^{d_1}}^{\tilde\kappa^2 \textbf{b},\textbf{r}}\big(u(\tilde\kappa\, \cdot\,)\big) = \tilde\kappa^2 (\atL{b}{r}{d_1}u)(\tilde\kappa \,\cdot\,)$, \eqref{eq:reduction-PDE-2} is equivalent to
\begin{equation}\label{eq:rescaling-2}
\Delta_{\R^{d_1}}^{\tilde \kappa^2 \textbf{b},\textbf{r}} \big(u(\tilde\kappa\, \cdot\,) \big) - \gamma^2\kappa^4 u(\tilde\kappa\, \cdot\,)
 = \gamma^2\kappa^2 f(\tilde\kappa \,\cdot\,).
\end{equation}
Moreover,
\begin{align*}
\frac{1}{d_1+1}+\frac{d_1}{p_{d_1}'} -\frac{d_1}{2}
& = \frac{1}{d_1+1}+d_1 \frac{d_1-1}{2(d_1+1)} -\frac{d_1}{2} \\
& = \frac{2}{d_1+1}-1.
\end{align*}
Substituting $u(\tilde\kappa\, \cdot\,)$ and $\tilde \kappa^2 f(\tilde\kappa\, \cdot\,)$ by new functions $\tilde u$ and $\tilde f$ shows that  \cref{eq:rescaling-1} and \cref{eq:rescaling-2}  are equivalent to \cref{eq:rescaling-3} and \cref{eq:PDE-reduced}, which finishes the proof.
\end{proof}

Similar to \cite[Lemma 2]{KoRi07} and \cite[Lemma 3.4]{KoTa05}, the local estimate of \cref{lem:reduction} is a consequence of the results in \cite{KoTa05a} by Koch and Tataru. To state their result, we need to introduce some notation. Given $\lambda>0$ and $j\in\N$, let $S_\lambda^j$ be the class of all symbols $a\in C^\infty(\R^n\times\R^n, \C)$ satisfying
\begin{equation}\label{eq:symbol-class}
|\partial_x^\alpha\partial_\xi^\beta a(x,\xi)|
\le C_{\alpha,\beta}
\begin{cases}
\lambda^{-|\beta|}& \text{for } |\alpha|\le j,\\
\lambda^{\frac{|\alpha|-j}{2}-|\beta|}& \text{for } |\alpha|> j.
\end{cases}
\end{equation}
For $\lambda\ge 1$, let $\lambda^k S_\lambda^j = \{\lambda^k a:a\in S_\lambda^j\}$. (Note that \cite{KoTa05a} requires $\lambda>1$, but this restriction is insignificant, as we will see in the sketch of the proof of \cref{thm:Koch-Tataru} below.) Given a symbol $a\in S_\lambda^j$, we denote by $a^w=a^w(x,D)$ the pseudo-differential operator defined by the Weyl calculus, that is,
\[
a^w(x,D) u(x) = (2\pi)^{-n} \int_{\R^n} \int_{\R^n} a\big(\tfrac 1 2 (x+y),\xi\big) e^{i\langle x-y, \xi\rangle } u(y) \, dy \, d\xi.
\]
We denote the characteristic set of a given symbol $p\in C^\infty(\R^n\times\R^n, \C)$ by
\[
\chr p=\{(x,\xi)\in \R^n\times(\R^n\setminus\{0\}):p(x,\xi)=0 \},
\]
and set $\Sigma={\chr p} \cap B^\lambda$, where $B^\lambda$ denotes the ball
\[
B^\lambda = \{(x,\xi)\in \R^n\times \R^n: |x|<1 \text{ and } |\xi|<\lambda \}.
\]
Then, for $|x|<1$, the $x$-section of $\Sigma$ is given by
\[
\Sigma_x = \{ \xi \in \R^n\setminus\{0\} :  |\xi|<\lambda \text{ and }  p(x,\xi) = 0 \}.
\]
We will use the following special case of Theorem 2.5~(i) from \cite{KoTa05a}:

\begin{theorem}\label{thm:Koch-Tataru}
Let $k\in \{0,\dots,n-2\}$ and
\begin{equation}\label{eq:exponents}
q = \frac{2(n+1-k)}{n-1-k} \quad\text{and}\quad \rho(q) = \frac{n-1+k}{2(n+1-k)}.
\end{equation}
Suppose that $p\in \lambda S_\lambda^2$ is a real symbol satisfying the following conditions:
\begin{enumerate}
\item[(i)] The symbol $p$ is of principal type, that is,
\[
|\nabla_\xi \, p(x,\xi)| \gtrsim 1 \quad \text{for all } (x,\xi)\in\Sigma.
\]
\item[(ii)] The set $\Sigma_x$ has $n-1-k$ non-vanishing curvatures for all $|x|<1$. More precisely, for all $|x|<1$ and $\xi\in \Sigma_x$, the second fundamental form of $\Sigma_x$ at $\xi$ admits an $(n-1-k)$-minor $M$ such that
\[
|{\det M}| \gtrsim \lambda^{k-n+1},
\]
where the constant is independent of $x$ and $\xi$.
\end{enumerate}
Let $\chi\in S_\lambda^0$ be compactly supported in $B^\lambda$. Then
\begin{equation}\label{eq:dispersive}
\lambda^{-\rho(q)}\norm{\chi^w u}_q \lesssim \norm{p^w u}_2 + \norm{u}_2. 
\end{equation}
\end{theorem}

An inspection of the proof of Theorem 2.5~(i) in \cite{KoTa05a}, whose parts can be found in Sections~3 and~4 there, shows that the constant in \eqref{eq:dispersive} depends only on the constants in \eqref{eq:symbol-class} and the constants in the conditions (i) and (ii) above. This observation is fundamental to the proof of \cref{thm:cluster}, where we apply \cref{thm:Koch-Tataru} to a whole family of operators. We provide a brief walkthrough of the relevant arguments in \cite{KoTa05a} to convince the reader of the uniformity of the estimates.

\begin{proof}[Sketch of the proof]
Following the arguments of \cite[Lemma 3.8 (i)]{KoTa05a}, one first reduces the operator $p^w$ to some canonical form. Given $(x_0,\xi_0)\in \Sigma$, we have $|\partial_{\xi_k}p(x,\xi)|\sim 1$ in an $\varepsilon B^\lambda$-neighborhood of $(x_0,\xi_0)$ for some $k\in\{1,\dots,n\}$. Similar to \cite[Lemma 3.6]{KoTa05a}, covering $B^\lambda$ by dilates of such balls and using a partition of unity, it suffices to prove \cref{eq:dispersive} with $\chi$ on the left-hand side replaced by a smooth cut-off function $\chi_\varepsilon$ supported on one of the $\varepsilon B^\lambda$-balls, and $p$ replaced by $p_\varepsilon=\chi_\varepsilon \, p$. Assuming without loss of generality that $k=1$, the implicit function theorem yields
\[
\{(x,\xi):p(x,\xi)=0\}=\{(x,\xi):\xi_1+a(x,\xi)=0\}
\]
on a given $\varepsilon B^\lambda$-ball, with $a\in \lambda S^2_\lambda$. Let $e$ be the symbol defined by
\[
e(x,\xi) = \tilde \chi_\varepsilon(x,\xi)\, \frac{\xi_1+a(x,\xi)}{p(x,\xi)},
\]
where $\tilde \chi_\varepsilon=1$ on the support of $\chi_\varepsilon$.
Since $|\partial_{\xi_1} p(x,\xi)|\sim 1$ on the chosen $\varepsilon B^\lambda$-ball, we have $e\in S^2_\lambda$, whence it suffices to show \cref{eq:dispersive} with $p_\varepsilon$ on the right-hand side replaced by $\tilde p_\varepsilon=e p_\varepsilon$, see \cite[Lemma 3.3]{KoTa05a}. Note that the bounds for the derivatives of $a$ and $e$ depend only on the corresponding bounds for the derivatives of $p$ and the constant in condition (i). Moreover, both the characteristic set and the conditions (i) and (ii) remain invariant when replacing the $p$ by the microlocalized symbol $\tilde p_\varepsilon$. In summary, it suffices to prove \cref{thm:Koch-Tataru} for the special case where $p$ is of the form
\[
p(x,\xi) = \xi_1 + a(x,\xi'),
\]
where $\xi = (\xi_1,\xi') \in \R^{n}$ and $a\in \lambda S_\lambda^2$.

We make a slight change in notation. As in the beginning of Section 4.1 of \cite{KoTa05a}, we write $t=x_1$ and denote henceforth the component $x'\in \R^{n-1}$ again by $x$. The idea is to interpret the first coordinate as a time parameter. With this notation, we are led to consider the operator $D_t + a^w$, where $D_t=-i\partial_t$. It suffices to prove
\begin{equation}\label{eq:dispersive-canonical}
\lambda^{-\rho(q)} \norm{\chi^w(x,D) u}_{L^q(\R^{n})} \lesssim \norm*{(D_t+a^w)u}_{L^2(\R^{n})} + \norm{u}_{L^2(\R^{n})}
\end{equation}
for any smooth cut-off function $\chi$ on $\R^{n-1}\times \R^{n-1}$ compactly supported where $|x|\le 1$ and $|\xi|\le \lambda$ and all operators $D_t + a^w$ whose symbol satisfies (i) and (ii). Via functional calculus, the (real) symbol $a$ generates the isometric evolution operator
\[
S(t,s)=\exp(-i(t-s)a^w),\quad t,s\in\R.
\]
The idea is to derive \cref{eq:dispersive-canonical} from appropriate decay estimates associated with the evolution operator $S(t,s)$, which in turn follow from corresponding $L^1$-$L^\infty$ decay estimates by the methods of Keel and Tao \cite{KeTa98}. To prove \cref{eq:dispersive-canonical}, it suffices to show the following special case of Proposition 4.8 in \cite{KoTa05a} (with $f_1=0$ and $r=s$ there), which is the estimate
\begin{equation}\label{eq:dispersive-0}
\lambda^{-\rho(q)} \norm{\chi^w(x,D) u}_{L^q(\R^{n})} \lesssim \norm{f}_{L^1([0,1],L^2(\R^{n-1}))} + \norm{u_0}_{L^2(\R^{n-1})},
\end{equation}
where $u$ solves the initial value problem
\begin{equation}\label{eq:IVP}
(D_t+a^w)u = f,\quad u(0,\cdot)=u_0
\end{equation}
on the interval $[0,1]$. Due to Duhamel's formula, the solution of \cref{eq:IVP} can be represented as
\[
u(t,\cdot) = S(t,0)u_0 + i\int_0^t S(t,s) f(s) \, ds.
\]
Thus, to show \cref{eq:dispersive-0}, it suffices to show that the operators $\chi^w S(t,s)$ are bounded from $L^2$ to $L^q$ with operator norm
\begin{equation}\label{eq:decay-0}
\norm{\chi^w S(t,s)}_{2\to q}\lesssim \lambda^{\rho(q)}.
\end{equation}
As argued in the proof of Proposition 4.8 of \cite{KoTa05a} (note that $d=n-1$ there), \cref{eq:decay-0} follows by a $TT^*$ argument and by the results of \cite{KeTa98} by Keel and Tao from the decay estimates
\[
\norm{\chi^w S(t,s)\chi^w}_{1\to \infty}\lesssim \lambda^{(n-1+k)/2} \left|t-s\right|^{-(n-1-k)/2}.
\]
An inspection of the arguments in \cite{KeTa98} shows that the decay estimates they derive actually depend only on the constant of the initial $L^1$-$L^\infty$ estimates, which up to the endpoint are the arguments of Section 3 there, while the estimate for the endpoint is treated in Sections 4 and 5 there.

Following the arguments of the proof of Proposition 4.7 of \cite{KoTa05a}, to prove the decay estimate
\begin{equation}\label{eq:decay}
\norm{S(t_0,0)\chi^w u_0}_{\infty}\lesssim \lambda^{(n-1+k)/2} \left|t_0\right|^{-(n-1-k)/2} \|u_0\|_1.
\end{equation}
at time $t_0$ (one can assume without loss of generality $s=0$), one passes to a normalized setup with new frequency parameter $\mu=\sqrt{t_0\lambda}$ by rescaling via
\[
(t,x,\xi)\mapsto \bigg(t_0 t,\frac{x\sqrt{t_0}}{\sqrt\lambda},\frac{\xi\sqrt\lambda}{\sqrt{t_0}}\bigg).
\]
Let $u(t,\cdot) := S(t,0)\chi^w u_0$. We consider the functions
\[
v(t,x):=u\bigg(t_0t,\frac{x\sqrt{t_0}}{\sqrt\lambda}\bigg)\quad\text{and} \quad v_0(x):=u_0\bigg(\frac{x\sqrt{t_0}}{\sqrt\lambda}\bigg).
\]
Then \cref{eq:decay} is equivalent to
\begin{equation}\label{1-infty}
\norm{v(1,\cdot)}_{L^\infty(\R^{n-1})} \lesssim \mu^k \norm{v_0}_{L^1(\R^{n-1})}.
\end{equation}
(In \cite{KoTa05a}, one actually has to replace the rescaling $t/t_0$ by $t_0 t$, since we want $v(1,x)=u(t_0,\frac{x\sqrt{t_0}}{\sqrt\lambda})$.) Let
\[
\tilde a(t,x,\xi) = t_0\, a\bigg(t_0t,\frac{x\sqrt{t_0}}{\sqrt\lambda},\frac{\xi\sqrt\lambda}{\sqrt{t_0}}\bigg)\quad\text{and}\quad
\tilde \chi(x,\xi) = \chi \bigg(\frac{x\sqrt{t_0}}{\sqrt\lambda},\frac{\xi\sqrt\lambda}{\sqrt{t_0}}\bigg).
\]
Note that
\[
v(0,x) 
 = (\chi^w u_0)\bigg(0,\frac{x\sqrt{t_0}}{\sqrt\lambda}\bigg) = \tilde \chi^w(x,D)v_0.
\]
Thus, the function $v$ solves the initial value problem
\[
(D_t+\tilde a^w(t,x,D))v=0,\quad v(0,\cdot)=\tilde \chi^w(x,D)v_0.
\]
To prove \cref{1-infty}, we write
\[
v(t,y) = \int_{\R^{n-1}} K(t,y,\tilde y) (\chi^w v_0)(\tilde y) \,d\tilde y.
\]
Using Proposition 4.3 of \cite{KoTa05a}, we want to derive an $L^\infty$-bound for the kernel~$K$. To that end, we need to introduce some further notation of \cite[p.\ 235]{KoTa05a}. For fixed $(x,\xi)$, let $(x^t,\xi^t)$ denote the Hamiltonian flow given by the solution of the initial value problem
\[
\frac{d}{dt} x^t = \tilde a_\xi (x^t,\xi^t) ,\quad \frac{d}{dt} \xi^t = - \tilde a_x(x^t,\xi^t) ,\quad (x^0,\xi^0) =  (x,\xi),
\]
where $\tilde a_\xi$ and $\tilde a_x$ denote the partial derivatives with respect to $\xi$ and $x$. Let $\psi(t,x,y)$ be the phase shift given by
\[
\frac{d}{dt} \psi(t,x,\xi) = (-\tilde a + \xi^t \tilde a_\xi) (x^t,\xi^t),\quad \psi(0,x,\xi) = 0.
\]
Then, using Proposition 4.3 of \cite{KoTa05a} for the special case $s=0$, we get the representation
\[
K(t,y,\tilde y) =  \int_{\R^{n-1}} \int_{\R^{n-1}} e^{-\frac{1}{2}(\tilde{y}-x)^2} e^{i\Phi(t, x, \xi, y)} \, G(t, x, \xi, y) \, dx\, d\xi
\]
for $t\in[0,1]$ and $y,\tilde y\in \R^{n-1}$, with
\[
\Phi (t, x, \xi, y)= - \xi(\tilde{y}-x)+\psi(t, x, \xi)+ \xi^t(y-x^t)
\]
and some function $G$ satisfying
\begin{equation}\label{eq:G}
\big|(x^t-y)^\gamma\partial_x^\alpha \partial_{\xi}^\beta \partial_y^\nu G(t, x, \xi, y)\big| \le C_{\gamma, \alpha, \beta, \nu},
\end{equation}
where the constant $C_{\gamma, \alpha, \beta, \nu}$ only depends on the corresponding bounds for the derivatives of $\tilde a$.

Returning to the arguments of the proof of Proposition 4.7 in \cite{KoTa05a}, we consider the ball
\[
\tilde B=\{(x,\xi):|x|\le \mu t_0^{-1},|\xi|\le \mu\}.
\]
Then, since $\supp \tilde \chi\subseteq \tilde B$, there is some constant $C>0$ such that
\[
\begin{aligned}
v(t,y) = \int_{\R^{3(n-1)}} \textbf{1}_{C\tilde{B}}(x,\xi) G(t, x, \xi, y) e^{-\frac{1}{2}(\tilde{y}-x)^2}e^{i\Phi(t, x, \xi, y)}  (\tilde{\chi}^w v_0)(\tilde{y}) \, dx \, d\xi \, d\tilde{y}.
\end{aligned}
\]
up to some error term that can be estimated by $C_N \mu^{-N}\|v_0\|_{L^1(\R^{n-1})}$, see \cite[p.~241]{KoTa05a}. Then \cref{1-infty} follows once we have shown the bound
\begin{equation}\label{eq:K-bound-1}
\int_{\R^{n-1}} \textbf{1}_{C\tilde{B}}(x,\xi) |G(1,x,\xi,y)| \,d\xi \lesssim\mu^k.
\end{equation}
By \cref{eq:G}, we may bound the above integral by a constant $C_N$ times
\begin{equation}\label{eq:K-bound-2}
 \sup_{|x|\le C \mu t_0^{-1}} \int_{\R^{n-1}} \textbf{1}_{C\tilde{B}}(x,\xi) (1+|x^t-y|)^{-N} \,d\xi.
\end{equation}
As in \cite[p.\ 241]{KoTa05a}, a linear approximation of $t\mapsto \partial x^t/\partial\xi$ shows that
\begin{equation}\label{eq:linearization}
\frac{\partial x^1}{\partial\xi} = \frac{\partial x^t}{\partial\xi}\bigg|_{t=1} = \tilde a_{\xi\xi} (0,x,\xi) + O(\sqrt{t_0}),
\end{equation}
where $\tilde a_{\xi\xi}$ denotes the Hessian matrix of $\xi\mapsto \tilde a(0,x,\xi)$. Now the curvature condition of (ii) yields that $\tilde a_{\xi\xi}(0,x,\xi)$ admits a non-degenerate $(n-1-k)$-minor, whence one may locally choose coordinates $\xi=(\xi',\xi'')$ with $\xi'=(\xi_1,\dots,\xi_{n-1-k})$ so that the matrix $\tilde a_{\xi'\xi'}(0,x,\xi)$ is non-degenerate. As in \cite[p.\ 242]{KoTa05a}, estimating the integral \cref{eq:K-bound-2} in $\xi''$ via the support of $C\tilde B$ and in $\xi'$ via \cref{eq:linearization}, we get
\[
\int_{C\tilde B} (1+|x^t-y|)^{-N} \,d\xi \lesssim \mu^k,
\]
which yields \cref{eq:K-bound-1} and thus \cref{1-infty}.
\end{proof}

With \cref{thm:Koch-Tataru} at hand, we now prove \cref{lem:reduction}, which completes the proof of the spectral cluster estimates of \cref{thm:cluster}. Recall that $\mathbf{B}\subseteq [0,\infty)^N$ denotes the compact subset of \cref{thm:cluster}.

\begin{proof}[Proof of \cref{lem:reduction}]
Let $B_1,B_2\subseteq\R^{d_1}$ denote the open balls of radius 1 and 2 centered at the origin.
We have to show that there is some $\gamma=\gamma_{\mathbf B,\mathbf r}>0$ such that, for all $\textbf{b}\in \textbf{B}$ and all $\kappa\ge 1$, we have
\begin{equation}\label{eq:rescaling-3b}
\kappa^{2/(d_1+1)} \| u|_{B_1} \|_{p_{d_1}'} \le C_{\mathbf B,\mathbf r}\, \kappa \, \norm*{ u|_{B_2}}_2 + \kappa^{-1} \norm*{ f|_{B_2}}_2
\end{equation}
whenever $u\in L^2(\R^{d_1})$ and $f\in L^2(\R^{d_1})$ satisfy
\[
\big(\Delta_{\R^{d_1}}^{\gamma^2\kappa^2 \mathbf{b},\mathbf{r}} - \gamma^2 \kappa^4 \big)u = f.
\]
Let $\kappa\ge 1$. The parameter $\gamma=\gamma_{\textbf{B},\textbf{r}}>0$ will be chosen later in the proof. Recall that $\mathbf b=(b_1,\dots,b_N) \in [0,\infty)^N$, $\mathbf r=(r_1,\dots,r_N)\in(\N\setminus\{0\})^N$, $N\in\N\setminus\{0\}$, and
\[
\atL{\gamma^2 \kappa^2 \textbf{b}}{r}{d_1} =
(-\Delta_{\R^{d_1-2|\textbf{r}|_1}}) \oplus \tL{\gamma^2 \kappa^2 b_1}{2r_1} \oplus \dots  \oplus \tL{\gamma^2 \kappa^2 b_N}{2r_N}.
\]
By definition, on each block, we have
\[
\tL{\gamma^2 \kappa^2 b_n}{2r_n} = -\Delta_z + \tfrac 1 4 \left(\gamma^2 \kappa^2 b_n\right)^2 |z|^2 - i \left(\gamma^2 \kappa^2 b_n\right) \omega_{\R^{2r_n}}(z,\nabla_z),\quad z \in \R^{2r_n},
\]
where $\omega_{\R^{2r_n}}(z,w)=\langle J_{\R^{2r_n}} z, w\rangle_{\R^{2r_n}}$ is the standard symplectic form induced by
\[
J_{\R^{2r_n}} = \begin{pmatrix}
0 & -\mathrm{id}_{\R^{r_n}} \\
\mathrm{id}_{\R^{r_n}} & 0 
\end{pmatrix}
\in \R^{2r_n\times 2r_n}.
\]
In the Weyl calculus, the symbol $\sigma^{\textbf{b}}=\sigma_{\kappa,\gamma}^{\textbf{b},\textbf{r}}$ of $\Delta^{\gamma^2\kappa^2 \textbf{b},\textbf{r}}_{\R^{d_1}}-\gamma^2\kappa^4$ is given by
\[
\sigma^{\textbf{b}}(x,\xi) = \left|\xi+\tfrac{1}{2} \gamma^2\kappa^2 J^{\textbf{b}} x\right|^2 - \gamma^2\kappa^4,\quad x,\xi\in \R^{d_1},
\]
where
\begin{equation}\label{eq:J^b}
J^{\textbf{b}} = 
\begin{pmatrix}
0 &0  &\dots & 0 \\
0 & b_1 J_{\R^{2r_1}} & \ddots  & \vdots \\
 \vdots & \ddots & \ddots & 0 \\
0 & \dots & 0 & b_N J_{\R^{2r_N}}
\end{pmatrix}
\in \R^{d_1\times d_1}.
\end{equation}
In the following, given $\lambda>0$, we use again the notation
\[
B^\lambda = \{(x,\xi)\in \R^{d_1}\times \R^{d_1} : |x|<1 \text{ and } |\xi|<\lambda \}.
\]
We decompose the phase space into a non-elliptic region, where we will apply \cref{thm:Koch-Tataru}, and an elliptic region, where even better estimates are available. 

\medskip

(1) \textit{The non-elliptic region.}
We apply \cref{thm:Koch-Tataru} with parameters
\[
k =0 \quad \text{and}\quad
\lambda = \kappa^2.
\]
Let $\eta_0:\R^{d_1}\times\R^{d_1}\to\C$ be a smooth cut-off function such that $\eta_0(x,\xi)=1$ for $(x,\xi)\in B^1$ and $\eta_0(x,\xi)=0$ for $(x,\xi)\notin 2 B^1$. We define $\eta$ via
\[
\eta(x,\xi)=\eta_0(x,\xi/\lambda).
\]
Let $p^{\textbf{b}}$ and $a^{\textbf{b}}$ be the symbols given by
\[
p^{\textbf{b}}(x,\xi) = \lambda^{-1} \big(\sigma^{\textbf{b}}\eta \big)(x,\xi)
\quad\text{and}\quad
a^{\textbf{b}}(x,\xi) = \lambda^{-2} \big(\sigma^{\textbf{b}}\eta \big)(x,\xi).
\]
Since $\sigma^{\textbf{b}}$ is a polynomial of degree 2, we have
\[
\partial_x^\alpha\partial_\xi^\beta \sigma^{\textbf{b}}(x,\xi) = 0\quad\text{for } |\alpha|+|\beta|>2.
\]
Moreover, for all $(x,\xi)\in B^\lambda$ and $|\alpha|+|\beta|\le 2$,
\[
\lambda^{-2} \,\big|\partial_x^\alpha\partial_\xi^\beta \sigma^{\textbf{b}}(x,\xi)\big|
 = \bigg|\partial_x^\alpha\partial_\xi^\beta \bigg( \Big|\frac{\xi}{\lambda}+\frac{1}{2} \gamma^2 J^{\textbf{b}}\, x\Big|^2 - \gamma^2 \bigg) \bigg| \\
 \lesssim_{\textbf{B},\textbf{r},\alpha,\beta} \lambda^{-|\beta|}.
\]
Hence, via Leibniz rule, we have in particular
\[
\big| \partial_x^\alpha\partial_\xi^\beta a^{\textbf{b}}(x,\xi) \big|
\le C_{\alpha,\beta}
\begin{cases}
\lambda^{-|\beta|}& \text{for } |\alpha|\le 2,\\
\lambda^{\frac{|\alpha|-2}{2}-|\beta|}& \text{for } |\alpha|> 2.
\end{cases}
\]
Thus, $p^{\textbf{b}}=\lambda a^{\textbf{b}}\in \lambda S_\lambda^2$, which matches the requirements of \cref{thm:Koch-Tataru}.

Next, to apply \cref{thm:Koch-Tataru}, we verify the conditions (i) and (ii) from there. We consider $\Sigma=\chr p^{\textbf{b}} \cap B^\lambda$. For fixed $x\in\R^{d_1}$, the set
\[
\{\xi \in \R^{d_1} : \sigma^{\textbf{b}}(x,\xi) = 0 \}
\]
is a sphere of radius $\gamma\kappa^2  = \gamma \lambda$ with center $-\tfrac 1 2 \gamma^2 \kappa^2  J^{\textbf{b}}x =-\tfrac 1 2 \gamma^2 \lambda J^{\textbf{b}} x\in \R^{d_1}$. Note that $\eta|_{B^\lambda}=1$. Thus, by choosing $\gamma=\gamma_{\textbf{B},\textbf{r}}>0$ small enough, we can ensure that
\[
\{(x,\xi) \in \R^{d_1} \times \R^{d_1}  : \sigma^{\textbf{b}}(x,\xi) = 0 \text{ and } |x|<1 \} \subseteq B^\lambda
\]
and that for $|x|<1$ the $x$-section $\Sigma_x$ is the whole sphere of radius $\gamma \lambda$ with center $-\tfrac 1 2 \gamma^2 \lambda J^{\textbf{b}} x\in \R^{d_1}$. For condition (i) of \cref{thm:Koch-Tataru}, we observe that $(x,\xi)\in \Sigma$ implies that $\sigma^{\textbf{b}}(x,\xi)=0$ since $\eta|_{B^\lambda}=1$. Thus,
\begin{align*}
\left| \nabla_\xi \, p^{\textbf{b}} (x,\xi) \right|
& = \lambda^{-1} \left| \left( \nabla_\xi \, \sigma^{\textbf{b}} (x,\xi)\right) \eta(x,\xi) + \sigma^{\textbf{b}} (x,\xi) \, \nabla_\xi \, \eta(x,\xi)  \right| \\
& = 2\lambda^{-1} \left|\xi- \tfrac{1}{2} \gamma^2\lambda J^{\textbf{b}} x \right| = 2\gamma
\qquad \text{for all } (x,\xi) \in \Sigma.
\end{align*}
On the other hand, condition (ii) of \cref{thm:Koch-Tataru} follows from the fact that for $|x|<1$ the sphere $\Sigma_x$ has $d_1-1$ non-vanishing curvatures. Note that the required bound associated with the second fundamental form of $\Sigma_x$ in condition (ii) actually does not depend on $\textbf{b}\in \textbf{B}$.

Now, for $k=0$, the exponents $q$ and $\rho(k)$ in \cref{eq:exponents} are given by
\[
q = \frac{2(d_1+1)}{d_1-1} = p_{d_1}' \quad\text{and}\quad \rho(q) = \frac{d_1-1}{2(d_1+1)}.
\]
To show \cref{eq:rescaling-3b}, suppose that $u\in L^2(\R^{d_1})$ and $f\in L^2(\R^{d_1})$ satisfy
\[
\big(\Delta_{\R^{d_1}}^{\gamma^2\kappa^2 \textbf{b},\textbf{r}} -\gamma^2 \kappa^4\big) u = f.
\]
Let $\psi:\R^{d_1}\to\C$ be a smooth cut-off function with $\psi|_{B_1}=1$ and compact support in $B_2$. Then, applying \cref{thm:Koch-Tataru} with $k=0$ and $\lambda=\kappa^2$ yields
\[
\lambda^{-\rho(q)} \norm{\chi^w (\psi u)}_{p_{d_1}'} \lesssim_{\textbf{B},\textbf{r}} \norm{\psi u}_2 + \norm{(p^{\textbf{b}})^w (\psi u)}_2.
\]
Multiplying both sides by $\lambda^{1/2}$, we obtain
\begin{equation}\label{eq:KoTa-applied}
\lambda^{1/(d_1+1)} \norm{\chi^w (\psi u)}_{p_{d_1}'} \lesssim_{\textbf{B},\textbf{r}} \lambda^{1/2} \| u|_{B_2} \|_2 + \lambda^{1/2} \norm{(p^{\textbf{b}})^w (\psi u)}_2.
\end{equation}
We show that
\begin{equation}\label{eq:KoTa-L2-term}
\lambda^{1/2} \norm{(p^{\textbf{b}})^w (\psi u)}_2 \lesssim_{\textbf{B},\textbf{r}} \lambda^{1/2} \norm*{ u|_{B_2} }_2 + \lambda^{-1/2} \norm*{ f|_{B_2} }_2.
\end{equation}
Recall that $p^{\textbf{b}} = \lambda^{-1} \sigma^{\textbf{b}}\eta$. We want to use the calculus of pseudo-differential operators. We use the symbol classes $S^m$, which are defined as follows: Given $m\in\R$, we say that $p_0\in C^\infty (\R^{d_1}\times\R^{d_1},\C)$ lies in $S^m$ if
\begin{equation}\label{eq:Hoermander-symbol}
|\partial_\xi^\alpha \partial_x^\beta p_0(x,\xi)|\lesssim_{\alpha,\beta}\langle \xi\rangle ^{m-|\alpha|}
\quad \text{for all }x,\xi\in \R^{d_1},\alpha,\beta \in\N^{d_1},
\end{equation}
where $\langle \xi \rangle = (1 + |\xi|^2)^{1/2}$. (Note that we have changed from $\smash{\partial_x^\alpha\partial_\xi^\beta}$ to $\smash{\partial_\xi^\alpha \partial_x^\beta}$ in our notation compared to the definition of the symbol classes $S^j_\lambda$ of \cite{KoTa05a}.) If $p_0\in S^m$ for all $m\in\R$, we write $p_0\in S^{-\infty}$.

Note that the symbols $\sigma^{\textbf{b}}$ and $\eta$ satisfy the bounds
\begin{align*}
\lambda^{-2} \,\big|\partial_\xi^\alpha\partial_x^\beta \sigma^{\textbf{b}}(x,\xi)\big|
 & = \bigg|\partial_\xi^\alpha\partial_x^\beta \bigg( \Big|\frac{\xi}{\lambda}+\frac{1}{2} \gamma^2 J^{\textbf{b}} x\Big|^2 - \gamma^2 \bigg) \bigg| \\
 & \lesssim_{\textbf{B},\textbf{r},\alpha,\beta}  \Big\langle \frac \xi \lambda \Big\rangle^{2-|\alpha|} \lambda^{-|\alpha|} \quad\text{for } |x|\le 1
\end{align*}
and 
\[
|\partial_\xi^\alpha \partial_x^\beta \eta (x,\xi)| \lesssim_{M,\alpha,\beta} \Big\langle \frac \xi \lambda \Big\rangle^{-M-|\alpha|} \lambda^{-|\alpha|} \quad \text{for all } M\in\N.
\]
In particular, we have $\lambda^{-2} \sigma^{\textbf{b}} \in S^2$ and $\eta\in S^{-\infty}$. Now, the calculus of pseudo-differential operators yields
\begin{equation}\label{eq:asymptotic-expansion}
\lambda^{-2} (\sigma^{\textbf{b}} \eta)^w = \lambda^{-2} \eta^w (\sigma^{\textbf{b}})^w + (\tau^{\textbf{b}})^w,    
\end{equation}
where the symbol $\tau^{\textbf{b}}\in S^{-\infty}$ satisfies the bounds
\begin{equation}\label{eq:asymptotic-error}
|\partial_\xi^\alpha \partial_x^\beta \tau^{\textbf{b}} (x,\xi) | \lesssim_{\textbf{B},\textbf{r},M,\alpha,\beta} \Big\langle \frac \xi \lambda \Big\rangle^{-M-|\alpha|-1} \lambda^{-|\alpha|-1}.
\end{equation}
For details, see Section 3 of Chapter 2 in \cite{Fo89} and in particular Theorem 2.49 of \cite[p.\ 107]{Fo89}. The bound \cref{eq:asymptotic-error} actually follows from an inspection of the proof in \cite{Fo89}. Note that we can pass to the frequency scale $\lambda=1$ by rescaling via $(x,\xi)\mapsto (\lambda x,\xi/\lambda)$.

We derive $L^2$ bounds for the operators $\eta^w$ and $(\tau^{\textbf{b}})^w$ following the arguments of the proof of Corollary 2.56 of \cite[p.\ 112]{Fo89}. (Note that this is essentially the statement of the Calderón--Vaillancourt theorem, but we give a direct argument here since our symbols are sufficiently regular.) Let $K_{(\tau^{\textbf{b}})^w}$ denote the integral kernel of $(\tau^{\textbf{b}})^w$, that is,
\[
K_{(\tau^{\textbf{b}})^w}(x,y) = (2\pi)^{-d_1} \int_{\R^{d_1}} \tau^{\textbf{b}}(\tfrac 1 2 (x+y),\xi) e^{i\langle x-y,\xi\rangle} \, d\xi.
\]
Then, via integration by parts,
\[
(x-y)^\alpha K_{(\tau^{\textbf{b}})^w}(x,y) = C_\alpha \int_{\R^{d_1}} \partial_\xi^\alpha \tau^{\textbf{b}}(\tfrac 1 2 (x+y),\xi) e^{i\langle x-y,\xi\rangle} \, d\xi.
\]
Using \cref{eq:asymptotic-error}, we obtain
\begin{align*}
|(x-y)^\alpha K_{(\tau^{\textbf{b}})^w}(x,y)| & \lesssim_{\textbf{B},\textbf{r},M,\alpha} \int_{\R^{d_1}} \Big\langle \frac \xi \lambda \Big\rangle^{-M-|\alpha|-1} \lambda^{-|\alpha|-1} \, d\xi \notag \\
& \sim \lambda^{-|\alpha|+d_1-1} \int_{\R^{d_1}} \big\langle \tilde \xi\big \rangle^{-M-|\alpha|-1} \, d\tilde\xi.
\end{align*}
Since the last integral converges if we choose $M>0$ large enough, we get
\[
|K_{(\tau^{\textbf{b}})^w}(x,y)| \lesssim_{\textbf{B},\textbf{r},M,\beta} \left(\lambda\left|x-y\right|\right)^{-\beta} \lambda^{d_1-1}\quad\text{for all } \beta\ge 0.
\]
Thus, Young's convolution inequality shows that $(\tau^{\textbf{b}})^w$ is bounded on $L^2$ with
\begin{equation}\label{eq:error-L^2-bounds}
\|(\tau^{\textbf{b}})^w\|_{2\to 2} \lesssim_{\textbf{B},\textbf{r}} \lambda^{-1}.
\end{equation}
Similarly, $\eta^w$ is bounded on $L^2$ with
\begin{equation}\label{eq:error-L^2-bounds-ii}
\|\eta^w\|_{2\to 2} \lesssim 1.
\end{equation}
Hence, the left-hand side of \cref{eq:KoTa-L2-term} can be dominated by
\begin{align}
\lambda^{1/2} \|(p^{\textbf{b}})^w (\psi u)\|_2
& = \lambda^{-1/2} \|(\sigma^{\textbf{b}}\eta)^w (\psi u)\|_2 \notag \\
& \le \lambda^{-1/2} \|\eta^w (\sigma^{\textbf{b}})^w (\psi u)\|_2 + 
\lambda^{3/2} \|(\tau^{\textbf{b}})^w (\psi u)\|_2 \notag \\
& \lesssim_{\textbf{B},\textbf{r}} \lambda^{-1/2} \|\eta^w (\sigma^{\textbf{b}})^w (\psi u)\|_2 + \lambda^{1/2} \|u|_{B_2}\|_2. \label{eq:KoTa-L2-term-ii}
\end{align}
To prove \cref{eq:KoTa-L2-term}, it remains to bound the first summand of \cref{eq:KoTa-L2-term-ii}. Recall that $\lambda=\kappa^2$ and $(\sigma^{\textbf{b}})^w u =f$ by assumption, where
\[
(\sigma^{\textbf{b}})^w=\Delta^{\gamma^2\kappa^2 \textbf{b},\textbf{r}}_{\R^{d_1}}-\gamma^2\kappa^4.
\]
We choose a smooth cut-off function $\smash{\tilde\psi}:\R^{d_1}\to\C$ with $\smash{\tilde\psi}|_{\supp\psi}=1$ and compact support in $B_2$. Via Leibniz rule, we obtain
\begin{equation}\label{eq:Leibniz}
(\sigma^{\textbf{b}})^w (\psi u) = f\psi - u\left( \Delta_{\R^{d_1}} + i\gamma^2\kappa^2 \langle J^{\textbf{b}} \cdot\,,\nabla\rangle \right) \psi - \langle\nabla u,\nabla \psi\rangle,
\end{equation}
where $J^{\textbf{b}}$ is the $d_1\times d_1$ matrix of \cref{eq:J^b}. For the first two summands, using that $\psi$ is supported in $B_2$, we observe that \cref{eq:error-L^2-bounds-ii} implies
\begin{equation}\label{eq:Leibniz-i}
\|\eta^w (f\psi) \|_2
\lesssim \norm*{ f|_{B_2} }_2  
\end{equation}
and, since $\kappa^2=\lambda$, we also have
\begin{equation}\label{eq:Leibniz-ii}
\big \| \eta^w \big( u\left( \Delta_{\R^{d_1}} + i\gamma^2\kappa^2 \langle J^{\textbf{b}} \cdot\,,\nabla\rangle \right) \psi \big) \big \|_2 
\lesssim_{\textbf{B},\textbf{r}} \lambda \, \norm*{u|_{B_2}}_2 .
\end{equation}
Let $\tilde u := u\tilde \psi$. Then, since $\tilde\psi|_{\supp\psi}=1$, we have
\[
\eta^w \langle\nabla u,\nabla \psi\rangle = \sum_{j=1}^{d_1} \eta^w ((\partial_j \psi) \partial_j \tilde u).
\]
Note that the $j$-th summand on the right-hand side is the concatenation of the pseudo-differential operators associated with the symbols $\eta$, $(x,\xi)\mapsto \partial_j \psi(x)$ and $(x,\xi)\mapsto i\xi_j$. Note that $\eta\in S^{-\infty}$, $\partial_j \psi\in S^0$, and $\big((x,\xi)\mapsto i\xi_j\big) \in S^1$. Since the concatenation of pseudo-differential operators is the pseudo-differential operator associated with their twisted product, we have
\[
\eta^w ((\partial_j \psi) \partial_j \tilde u) = \omega_j^w \tilde u,
\]
where $\omega_j\in S^{-\infty}$ satisfies the bounds
\[
|\partial_\xi^\alpha \partial_x^\beta \omega_j (x,\xi) | \lesssim_{M,\alpha,\beta} \Big\langle \frac \xi \lambda \Big\rangle^{-M-|\alpha|+1} \lambda^{-|\alpha|+1},
\]
see \cite[Theorem 2.47]{Fo89}. Similar arguments as we used for \cref{eq:error-L^2-bounds} show that
\[
\|\omega_j^w\|_{2\to 2} \lesssim \lambda.
\]
Thus, we obtain 
\begin{equation}\label{eq:Leibniz-iii}
\|\eta^w \langle\nabla u,\nabla \psi\rangle\|_2 \lesssim  \lambda \, \|\tilde u\|_2 \lesssim_{\tilde\psi} \lambda \, \norm*{u|_{B_2}}_2 
\end{equation}
Combining the estimates \cref{eq:Leibniz-i}, \cref{eq:Leibniz-ii}, and \cref{eq:Leibniz-iii}, we get
\[
\|\eta^w (\sigma^{\textbf{b}})^w (\psi u) \|_2 \lesssim_{\textbf{B},\textbf{r}} \norm*{ f|_{B_2} }_2 + \lambda \, \norm*{u|_{B_2}}_2.
\]
Together with with \cref{eq:KoTa-L2-term-ii}, we obtain \cref{eq:KoTa-L2-term}. Then, using \cref{eq:KoTa-applied} together with \cref{eq:KoTa-L2-term} finally yields
\[
\lambda^{1/(d_1+1)} \norm*{\chi^w (\psi u)}_{p_{d_1}'}
 \lesssim_{\textbf{B},\textbf{r}} \lambda^{1/2} \norm*{ u|_{B_2} }_2 + \lambda^{-1/2} \norm*{ f|_{B_2} }_2 .
\]
Thus, we are done with the proof once we have verified that
\begin{equation}\label{eq:elliptic}
\lambda^{1/(d_1+1)} \norm*{(1-\chi^w) (\psi u)}_{p_{d_1}'} \lesssim_{\textbf{B},\textbf{r}} \lambda^{1/2} \norm{u|_{B_2}}_2 + \lambda^{-1/2} \norm{f|_{B_2}}_2.
\end{equation}

(2) \textit{The elliptic region.} The estimate \cref{eq:elliptic} follows from the ellipticity in the region of the phase space where $|x|\le 1$ and $|\xi|\ge \lambda$. Even though this estimate may not be surprising to experts, we provide a brief sketch of a proof for the convenience of the reader. The following arguments are essentially borrowed from \cite[Chapter 2]{Fo89}, but they can also be found in other places in the literature, for instance \cite{Ta74,So93}.

First, by choosing the constant $\gamma=\gamma_{\textbf{B},\textbf{r}}>0$ smaller if necessary, we can assume that $\gamma<1/10$ and
\[
\Big|\frac{1}{2} \gamma^2 J^{\textbf{b}} x\Big| \le \frac 1 {10} \, |x|\quad\text{for all } |x|\le 1 \text{ and } \textbf{b} \in \textbf{B}.
\]
Then, if $|x|\le 1$ and $|\xi|\ge \lambda$, we have
\[
|a^{\textbf{b}}(x,\xi)|
 = \bigg| \frac{\xi}{\lambda} + \frac{1}{2} \gamma^2 J^{\textbf{b}}\, \bar x \bigg|^2 - \gamma^2
 \sim 1 + \Big| \frac{\xi}{\lambda} \Big|^2 = \Big\langle \frac{\xi}{\lambda} \Big\rangle^2.
\]
Then one can argue as follows. Note that $d_1=r_0+2\left(r_1+\dots+r_N\right)\ge 2$, so in particular $p_{d_1}'<\infty$. Thus, if we put
\[
s := d_1\Big(\frac 1 2 -\frac 1 {p_{d_1}'}\Big)
   = d_1\Big(\frac 1 2 -\frac {d_1-1} {2(d_1+1)}\Big)
   = \frac {d_1} {d_1+1},
\]
then, by Sobolev's embedding theorem,
\begin{align}
\norm{(1-\chi^w)(\psi u) }_{p_{d_1}'}
& \lesssim \norm{(1-\chi^w)(\psi u)}_{L_s^2} \notag \\
& = \norm{(1-\Delta)^{s/2}(1-\chi^w)(\psi u)}_2 \notag \\
& \le \lambda^s \norm[\Big]{\Big(1-\frac{\Delta}{\lambda^2}\Big)^{s/2}(1-\chi^w)(\psi u)}_2.\label{eq:elliptic-sobolev}
\end{align}
We can assume without loss of generality that the symbol $\chi$ of \cref{thm:Koch-Tataru} satisfies $\chi(x,\xi)=\chi_0(x,\xi/\lambda)$ for some smooth cut-off function $\chi_0$ which is compactly supported in the ball $B^1$. Note that
\begin{align*}
\Big|\partial_\xi^\alpha \Big\langle \frac{\xi}{\lambda} \Big\rangle^s \Big|
& \lesssim_\alpha \Big\langle \frac \xi \lambda \Big\rangle^{s-|\alpha|} \lambda^{-|\alpha|},\\
|\partial_\xi^\alpha \partial_x^\beta  (1-\chi(x,\xi)) | & \lesssim_{\alpha,\beta} \Big\langle \frac \xi \lambda \Big\rangle^{-|\alpha|} \lambda^{-|\alpha|}.
\end{align*}
In particular, $\left\langle \frac \cdot \lambda \right\rangle\in S^s$ and $1-\chi\in S^0$, where $S^m$ denotes again the symbol class as defined in \cref{eq:Hoermander-symbol}. Let the symbol $\omega$ be given by
\[
\omega(x,\xi) := \Big\langle \frac \xi \lambda \Big\rangle^{s} \left(1-\chi(x,\xi)\right) \psi(x).
\]
Then $\omega\in S^s$ and
\[
|\partial_\xi^\alpha \partial_x^\beta \omega(x,\xi) |
 \lesssim_{\alpha,\beta} \Big\langle \frac \xi \lambda \Big\rangle^{s-|\alpha|} \lambda^{-|\alpha|}.
\]
Similar to \cref{eq:asymptotic-expansion}, the calculus of pseudo-differential operators yields
\begin{equation}\label{eq:asymptotic-A}
\Big(1-\frac{\Delta}{\lambda^2}\Big)^{s/2}\left(1-\chi^w\right)\psi = \omega^w + \rho_1^w,
\end{equation}
where $\rho_1 \in S^{s-1}$ is an error term with
\begin{equation}\label{eq:bounds-rho-1}
|\partial_\xi^\alpha \partial_x^\beta \rho_1 (x,\xi) |
 \lesssim_{\alpha,\beta} \Big\langle \frac \xi \lambda \Big\rangle^{s-|\alpha|-1} \lambda^{-|\alpha|-1}.
\end{equation}
Let $\tau^{\textbf{b}}$ be the symbol given by
\[
\tau^{\textbf{b}}(x,\xi )
 = \frac{\omega(x,\xi )}{\sigma^{\textbf{b}}(x,\xi )}\\
 = \lambda^{-2} \, \frac{\omega(x,\xi )}{|\frac{\xi}{\lambda}-\frac{1}{2} \gamma^2 J^{\textbf{b}} x|^2 - \gamma^2}.
\]
Then $\tau^{\textbf{b}}\in S^{s-2}$ and
\begin{equation}\label{eq:bounds-tau}
|\partial_\xi^\alpha \partial_x^\beta \tau^{\textbf{b}}(x,\xi ) |
 \lesssim_{\alpha,\beta} \Big\langle \frac \xi \lambda \Big\rangle^{s-|\alpha|-2} \lambda^{-|\alpha|-2}.
\end{equation}
As in the first part of the proof, let $\smash{\tilde\psi}:\R^{d_1}\to\C$ be a smooth cut-off function with $\smash{\tilde\psi}|_{\supp\psi}=1$ and compact support in $B_2$. Moreover, we may assume that there is another smooth cut-off function $\psi_0:\R^{d_1}\to\C$ such that $\smash{\psi_0}|_{\supp\psi}=1$ and $\smash{\tilde \psi}|_{\supp\psi_0}=1$. Then we have in particular $\supp\psi\subseteq\supp \psi_0 \subseteq \supp \smash{\tilde \psi}$.

Then, since $\smash{\psi_0}|_{\supp\psi}=1$, we have $\omega \psi_0=\omega$ and thus $\tau^{\textbf{b}} \psi_0=\tau^{\textbf{b}}$. Thus, again by the calculus of pseudo-differential operators, we can write
\begin{equation}\label{eq:asymptotic-B}
\omega^w = (\tau^{\textbf{b}})^w \psi_0 (\sigma^{\textbf{b}})^w + (\rho_2^{\textbf{b}})^w,
\end{equation}
where $\rho_2^{\textbf{b}}\in S^{s-1}$ is a symbol such that
\begin{equation}\label{eq:bounds-rho-2}
|\partial_\xi^\alpha \partial_x^\beta \rho_2^{\textbf{b}}(x,\xi) |
 \lesssim_{\textbf{B},\textbf{r},\alpha,\beta} \Big\langle \frac \xi \lambda \Big\rangle^{s-|\alpha|-1} \lambda^{-|\alpha|-1}.
\end{equation}
Similar to \cref{eq:error-L^2-bounds}, using the bounds \cref{eq:bounds-tau}, we see that the operator $(\tau^{\textbf{b}})^w$ is bounded on $L^2$. However, note that
\begin{align*}
|(x-y)^\alpha K_{(\tau^{\textbf{b}})^w}(x,y)| & \lesssim_{\textbf{B},\textbf{r},\alpha} \int_{\R^{d_1}} \Big\langle \frac \xi \lambda \Big\rangle^{s-|\alpha|-2} \lambda^{-|\alpha|-2} \, d\xi \\
& \sim \lambda^{-|\alpha|+d_1-2} \int_{\R^{d_1}} \big\langle \tilde \xi\big \rangle^{s-|\alpha|-2} \, d\tilde\xi
\end{align*}
and that the last integral converges if $|\alpha|>d_1+s-2$. Since $s=d_1/(d_1+1)<1$, we get
\begin{equation}\label{eq:integral-ops}
|K_{(\tau^{\textbf{b}})^w}(x,y)| \lesssim_{\textbf{B},\textbf{r},\beta} \left(\lambda\left|x-y\right|\right)^{-\beta} \lambda^{d_1-2}\quad\text{for all } \beta\ge d_1-1.
\end{equation}
Altogether, by Young's inequality, we obtain
\begin{equation}\label{eq:elliptic-errors}
\norm{(\tau^{\textbf{b}})^w}_{2\to 2} \lesssim_{\textbf{B},\textbf{r}} \lambda^{-2}.
\end{equation}
Moreover, one can show that $\rho_1^w$ and $(\rho_2^{\textbf{b}})^w$ are bounded on $L^2$ with
\begin{equation}\label{eq:elliptic-main}
\norm{\rho_1^w}_{2\to 2} \lesssim \lambda^{-1}\quad\text{ and } \quad
\norm{(\rho_2^\textbf{b})^w}_{2\to 2} \lesssim_{\textbf{B},\textbf{r}} \lambda^{-1}.
\end{equation}
However, this does not follow immediately as in the arguments above from the bounds \cref{eq:bounds-rho-1} and \cref{eq:bounds-rho-2}, since we only get \cref{eq:integral-ops} for $\beta\ge d_1$, in which case the function $x\mapsto |x|^{-\beta}$ is not locally integrable. However, since both symbols are in particular in the class $S^0$, they satisfy the regularity assumptions of the Calderón--Vaillancourt theorem, see for instance \cite[Chapter 2, Section 5]{Fo89}, and the bounds \cref{eq:elliptic-main} follow, for example, by applying the Calderón--Vaillancourt theorem to the operators $\lambda \rho_1^w$ and $\lambda(\rho_2^\textbf{b})^w$ and by using the rescaling $(x,\xi)\mapsto (\lambda x,\xi/\lambda)$.

Let $\tilde u := u\smash{\tilde \psi}$. Recall that $s=d_1/(d_1+1)$. Using \cref{eq:asymptotic-A} and \cref{eq:asymptotic-B}, and gathering the estimates \cref{eq:elliptic-sobolev}, \cref{eq:elliptic-errors} and \cref{eq:elliptic-main}, we obtain
\begin{align}
\lambda^{1/(d_1+1)} \norm{(1-\chi^w)(\psi u)}_{p_{d_1}'}
& \lesssim \lambda \, \norm[\Big]{\Big(1-\frac{\Delta}{\lambda^2}\Big)^{s/2}(1-\chi^w)(\psi u)}_2 \notag \\
 & \lesssim \lambda \, \norm{(\omega^w+\rho_1^w)\tilde u}_2 \notag \\
 & \le \lambda \left( \norm{(\tau^{\textbf{b}})^w \psi_0 (\sigma^{\textbf{b}})^w \tilde u}_2 + \norm{(\rho_2^{\textbf{b}})^w \tilde u}_2 + \norm{\rho_1^w \tilde u}_2 \right) \notag \\
 & \lesssim_{\textbf{B},\textbf{r}} \lambda^{-1} \norm{ \psi_0 (\sigma^{\textbf{b}})^w \tilde u}_2 + \norm{u|_{B_2}}_2.  \label{eq:elliptic-last}
\end{align}
Similar to \cref{eq:Leibniz}, we have
\[
(\sigma^{\textbf{b}})^w (\tilde\psi u) = f\psi - u\left( \Delta_{\R^{d_1}} + i\gamma^2\kappa^2 \langle J^{\textbf{b}} \cdot\,,\nabla\rangle \right) \tilde\psi - \langle\nabla u,\nabla \tilde\psi\rangle.
\]
Since $\smash{\tilde\psi}|_{\supp\psi_0}=1$, we have $\psi_0 \langle\nabla u,\nabla \smash{\tilde\psi}\rangle=0$. Thus, using similar bounds as in \cref{eq:Leibniz-i} and \cref{eq:Leibniz-ii}, that is,
\[
\|f\tilde \psi \|_2
\lesssim \norm*{ f|_{B_2} }_2  
\]
and
\[
\big \|  u\,\big( \Delta_{\R^{d_1}} + i\gamma^2\kappa^2 \langle J^{\textbf{b}} \cdot\,,\nabla\rangle \big) \psi  \big \|_2 
\lesssim_{\textbf{B},\textbf{r}} \lambda \, \norm*{u|_{B_2}}_2,
\]
we see that the last line of \cref{eq:elliptic-last} is in particular bounded by the right-hand side of \cref{eq:elliptic}, which finishes the proof.
\end{proof}

\section{Truncated restriction type estimates} \label{sec:restriction}

In this section, we prove \cref{thm:restriction-type}. Recall that $L=-(X_1^2+\dots+X_{d_1}^2)$ denotes a sub-Laplacian on an arbitrary two-step stratified Lie group $G$, where $X_1,\dots,X_{d_1}$ is a basis of the first layer of the stratification $\g=\g_1\oplus\g_2$. We write again $U=|\textbf{U}|^{1/2}$ for the square root of the Laplacian on $\g_2$, where $\textbf{U}$ is the vector of differential operators $\textbf{U}=(-iU_1,\dots,-iU_{d_2})$ and $U_1,\dots,U_{d_2}$ is a basis of $\g_2$. The restriction type estimates of \cref{thm:restriction-type} are stated in terms of the norms $\Vert\cdot\Vert_{M,2}$ given by 
\[
\norm{F}_{M,2} = \bigg(\frac{1}{M} \sum_{K\in\Z} \,\sup _{\lambda \in [\frac{K-1}{M}, \frac{K}{M})}|F(\lambda)|^2\bigg)^{1 / 2},\quad M\in(0,\infty).
\]
By \cite[Lemma~3.4]{ChHeSi16}, for every bounded Borel function $F:\R\to\C$,
\begin{equation}\label{eq:norm}
\|F\|_{L^2} \le \norm{F}_{M,2} \le C_{s} \big( \norm{F}_{L^2} + M^{-s} \norm{F}_{L^2_s} \big)\quad \text{for } s>1/2, 
\end{equation}
whence $\Vert\cdot\Vert_{M,2}$ is stronger than the $L^2$-norm.

To prove \cref{thm:restriction-type}, we need to show that
\begin{equation}\label{eq:restriction-type-0}
\norm{ F(L)\chi(2^\ell U) }_{p\to 2}
\le C_{p,\chi} 2^{-\ell d_2(\frac 1 p - \frac 1 2)} \|F\|_2^{1-\theta_p} \norm{F}_{2^{\ell},2}^{\theta_p} \quad \text{for all }\ell\in\Z
\end{equation}
and all $1\le p\le \min\{p_{d_1},p_{d_2}\}=\min\{2\frac{d_1+1}{d_1+3},2\frac{d_2+1}{d_2+3}\}$, where $\theta_p\in [0,1]$ satisfies
\[
\frac 1 p = (1-\theta_p) + \frac{\theta_p}{\min\{p_{d_1},p_{d_2}\}}.
\]
As our notation indicates, \cref{eq:restriction-type-0} is derived by interpolation between the endpoints $p=1$ and $p=\min\{p_{d_1},p_{d_2}\}$. The case $p=1$, where we have only the $L^2$-norm on the right-hand side, follows from a Plancherel estimate for the integral kernel (see also Section III.5 of \cite{ChOuSiYa16}).

\begin{lemma}\label{lem:restriction-type-p=1}
The restriction type estimate \cref{eq:restriction-type-0} holds true for $p=1$, i.e.,
\[
\norm{ F(L)\chi(2^\ell U) }_{1\to 2} \le C_{\chi} 2^{-\ell d_2/2} \|F\|_2 \quad \text{for all }\ell\in\Z.
\]
\end{lemma}

\begin{proof}
Let $\mathcal K_\ell$ denote the convolution kernel of the operator $F(L)\chi(2^\ell U)$. Then
\[
\| F(L)\chi(2^\ell U) f \|_2 = \| f * \mathcal K_\ell \|_2 \le  \| f \|_1  \| \mathcal K_\ell \|_2.
\]
By \cite[Corollary 8]{MaMue14b}, using the notation of \cref{prop:conv-kernel}, 
\begin{equation}
\begin{split}
\| \mathcal K_\ell \|_2^2
= (2 \pi)^{|\textbf{r}|_1-(d_1+d_2)} \int_{\mathfrak{g}_{2, r}^*} & \int_0^\infty \sum_{\textbf{k} \in \mathbb{N}^N}\left|F(s+\eigvp{k}{\mu})\chi(2^\ell|\mu|)\right|^2 \\
& \times \prod_{n=1}^N\left[\left(b_n^\mu\right)^{r_n}
\binom{k_n+r_n-1}{k_n}\right] d \sigma_{r_0}(s) \,d \mu, \label{eq:rest-interp-1}
\end{split}
\end{equation}
where $\sigma_{r_0}$ is the Dirac delta at 0 if $r_0 =0$, and
\[
d\sigma_{r_0}(s)=\frac{\pi^{r_0 / 2}}{\Gamma(r_0 / 2)} s^{r_0 / 2} \frac{d s}{s}\quad \text{if } r_0>0.
\]
Recall from \cref{prop:conv-kernel} that
\[
\eigvp{k}{\mu} = \sum_{n=1}^N \left(2k_n+r_n\right)b_n^\mu,
\]
where the function $\mu\mapsto\textbf{b}^\mu=(b_1^\mu,\dots,b_N^\mu)\in [0,\infty)^N$ is homogeneous of degree 1, see \cref{prop:rotation}. Using on the one hand that
\[
b_n^\mu \left(k_n+1\right) \le \eigvp{k}{\mu} \lesssim_A 1
\]
whenever $s+\eigvp{k}{\mu}\in\supp F\subseteq A$, where $A\subseteq (0,\infty)$ is the fixed compact subset of \cref{thm:restriction-type}, and using on the other hand
\[
\binom{k_n+r_n-1}{k_n} \sim (k_n+1)^{r_n-1},
\]
the right-hand side of \cref{eq:rest-interp-1} can be bounded by a constant times
\begin{equation}\label{eq:Plancherel-computation}
\int_0^\infty \int_{\mathfrak{g}_{2, r}^*} \sum_{\textbf{k} \in \mathbb{N}^N}\left|F(s+\eigvp{k}{\mu})\chi(2^\ell|\mu|)\right|^2 \prod_{n=1}^Nb_n^\mu \,d \mu \,d \sigma_{r_0}(s).
\end{equation}
We rewrite the integral over $\mu$ in polar coordinates, that is,
\[
\mu = \rho \omega \quad\text{for } \rho \in [0,\infty) \text{ and } |\omega|=1.
\]
Then, since $\mu\mapsto\textbf{b}^\mu$ is homogeneous of degree 1,
\[
\eigvp{k}{\mu} = \rho \eigvp{k}{\omega}.
\]
Thus, using $\rho=|\mu|\sim 2^{-\ell}$, \cref{eq:Plancherel-computation} is bounded by a constant times
\begin{align*}
2^{-\ell (d_2+N)}\int_0^\infty \int_{S^{d_2-1}}  \int_0^\infty \sum_{\textbf{k} \in \mathbb{N}^N} & \left|F(s+\rho\eigvp{k}{\omega})\chi(2^\ell\rho)\right|^2 \\
& \times \prod_{n=1}^N b_n^\omega \,\frac{d\rho}{\rho} \,d\sigma(\omega) \,d \sigma_{r_0}(s).
\end{align*}
Substituting $\rho=(\eigvp{k}{\omega})^{-1}\lambda$ in the inner integral, we see that the above term equals
\begin{align*}
2^{-\ell (d_2+N)}
\int_0^\infty \int_{S^{d_2-1}} \int_0^\infty \sum_{\textbf{k} \in \mathbb{N}^N}\left|F(s+\lambda)\chi(2^\ell (\eigvp{k}{\omega})^{-1} \lambda)\right|^2\\
\times \prod_{n=1}^Nb_n^{\omega} \,\frac{d \lambda}\lambda \,d\sigma(\omega) \,d \sigma_{r_0}(s).
\end{align*}
If $2^\ell (\eigvp{k}{\omega})^{-1} \lambda\in\supp\chi$, then
\[
(2k_n+r_n) b_n^\omega \le \eigvp{k}{\omega} \sim 2^\ell \lambda.
\]
Thus, $k_n\lesssim 2^\ell \lambda (b_n^\omega)^{-1}$ for the non-vanishing summands in the above sum over $\textbf{k}$. Hence, the above term is bounded by a constant times
\[
2^{-\ell d_2}
\int_0^\infty \int_0^\infty \left|F(s+\lambda)\right|^2 \lambda^N \frac{d \lambda}\lambda \,d \sigma_{r_0}(s),
\]
which is comparable to $2^{-\ell d_2} \|F\|^2_2$ since $F$ is compactly supported.
\end{proof}

It remains to show \cref{eq:restriction-type-0} for $p=\min\{p_{d_1},p_{d_2}\}$. Our arguments will show that
\begin{equation}\label{eq:restriction-type-i}
\norm*{ F(L)\chi(2^\ell U) }_{p\to 2}
\le C_{p,\chi} 2^{-\ell d_2(\frac 1p - \frac 1 2)} \norm*{F}_{2^{\ell},2} \quad \text{for all }\ell\in\Z
\end{equation}
and all $1\le p\le \min\{p_{d_1},p_{d_2}\}$, so \cref{eq:restriction-type-0} is in fact just \cref{eq:restriction-type-i} enhanced by interpolation with the $L^1$-$L^2$ estimate of \cref{lem:restriction-type-p=1}, where we have the $L^2$-norm of the multiplier instead of the Cowling--Sikora norm $\|\cdot\|_{2^\ell,2}$ on the right-hand side.

\begin{remark}\label{rem:ell_0}
The support conditions on the multipliers $F$ and $\chi$ actually imply
\begin{equation}\label{eq:ell_0}
F(L)\chi(2^\ell U)=0 \quad \text{for all } \ell<-\ell_0,
\end{equation}
where $\ell_0\in\N$ depends on the matrices $J_\mu$ of \cref{eq:skew-form-ii}, the inner product $\langle\cdot,\cdot\rangle$ on $\g$, and the compact subset $A\subseteq (0,\infty)$ of \cref{thm:restriction-type} where $F$ is supported, which means that the restriction type estimate \cref{eq:restriction-type-i} is trivial for $\ell<-\ell_0$. More precisely, using the formula \cref{eq:conv-kernel} of \cref{prop:conv-kernel} for the convolution kernel $\mathcal K_{F(L)\chi(2^\ell U)}$ of $F(L)\chi(2^\ell U)$, we have
\begin{align*}
\mathcal K_{F(L)\chi(2^\ell U)}(x,u) = & (2\pi)^{-r_0-d_2} \int_{\g_{2,r}^*}   \int_{\R^{r_0}} \sum_{\mathbf{k}\in\N^N} F(|\tau|^2+\eigvp{k}{\mu})\chi(2^\ell|\mu|) \\
& \times \bigg[\prod_{n=1}^N \varphi_{k_n}^{(b_n^\mu,r_n)}(R_\mu^{-1} P_n^\mu x)\bigg] e^{i\langle \tau, R_\mu^{-1} P_0^\mu x\rangle} \, e^{i\langle \mu, u\rangle} \, d\tau  \, d\mu,
\end{align*}
Recall from \cref{prop:conv-kernel} that
\[
\eigvp{k}{\mu} = \sum_{n=1}^N \left(2k_n+r_n\right)b_n^\mu,
\]
and from \cref{prop:rotation} that the function $\mu\mapsto\textbf{b}^\mu=(b_1^\mu,\dots,b_N^\mu)\in [0,\infty)^N$ is homogeneous of degree 1. Now, if $|\tau|^2+\eigvp{k}{\mu}\in \supp F\subseteq A$ and $2^\ell|\mu|\in\supp \chi$, then
\begin{equation}\label{eq:ell_0-spectral}
1 \gtrsim_A  |\tau|^2+\eigvp{k}{\mu}
\ge \eigvp{k}{\mu}
\ge \sum_{n=1}^N r_n b_n^\mu \\
= |\mu| \sum_{n=1}^N r_n b_n^{\bar \mu}
\sim 2^{-\ell} \sum_{n=1}^N r_n b_n^{\bar \mu},
\end{equation}
where $\bar \mu = |\mu|^{-1}\mu$. Due to \cref{eq:spectral-decomp}, we have
\begin{equation}\label{eq:Hilbert-Schmidt}
\sum_{n=1}^N r_n b_n^{ \mu} \sim_{\textbf{r}} \bigg(\sum_{n=1}^N 2 r_n \big(b_n^{ \mu}\big)^2 \bigg)^{1/2} = (\tr(J_{ \mu}^* J_{ \mu}))^{1/2}.
\end{equation}
As the latter expression is the pullback of the Hilbert-Schmidt norm via the injective map $\mu\mapsto J_\mu$, the left-hand side of \cref{eq:Hilbert-Schmidt} does not vanish for all $\mu\neq 0$. On the other hand, by \cref{prop:rotation}, all maps $\mu\mapsto b_n^\mu$ are continuous on $\g_2^*$. Thus \cref{eq:ell_0-spectral} implies $2^{-\ell}\lesssim_A 1$, which yields \cref{eq:ell_0}. Note that the constants occurring in \cref{eq:ell_0-spectral} and \cref{eq:Hilbert-Schmidt} depend only on the spectral properties of the matrices $J_\mu$, the inner product $\langle\cdot,\cdot\rangle$, and the set $A\subseteq (0,\infty)$. 
\end{remark}

By \cref{prop:joint-calculus}, we have
\begin{equation}\label{eq:joint-calculus-ii}
\big(F(L)\chi(2^\ell U)f\big)^\mu = F(L^\mu)\chi(2^\ell |\mu|) f^\mu
\end{equation}
for all $\mu\in\g_2^*$ and all Schwartz functions $f$ on $G$, where $f^\mu$ denotes the $\mu$-section of the partial Fourier transform along $\g_2$ and $L^\mu$ is the $\mu$-twisted Laplacian on $\g_1$. The idea is to apply a spectral cluster estimate for the twisted Laplacian $L^\mu$ and the classical Stein--Tomas restriction estimate \cite{To75} on the second layer $\g_2$. The spectral cluster estimate for $L^\mu$ is derived from the spectral cluster estimate \cref{eq:cluster} of the previous section.

\begin{lemma}
If $1\le p \le p_{d_1}$, then
\begin{equation}\label{eq:cluster-rescaled}
\norm*{\mathbf{1}_{[K|\mu|,(K+1)|\mu|)}(L^\mu) }_{p\to 2} \le C_p\, |\mu|^{\frac{d_1}2(\frac 1 p - \frac 1 2)}(K+1)^{\frac{d_1}2(\frac 1 p - \frac 1 2)-\frac 1 2}
\end{equation}
for all $K\in \N$ and almost all $\mu\in\g_2^*$.
\end{lemma}

\begin{proof}
Let $N\in\N\setminus\{0\}$, $r_0\in\N$, $\textbf{r}=(r_1,\dots,r_N)\in (\N\setminus\{0\})^N$, $\textbf{b}^\mu=(b_1^\mu,\dots,b_N^\mu)\in [0,\infty)^N$, and $\mu\mapsto R_\mu\in O(d_1)$ be as in \cref{prop:rotation}. By \cref{eq:rotation}, we have
\begin{equation}\label{eq:rotation-iii}
\big(F(L^\mu)\phi\big)\circ R_\mu = F\big(\atL{b^\mu}{r}{d_1}\big)(\phi\circ R_\mu)
\end{equation}
for almost all $\mu$ lying in a (non-empty) Zariski-open subset $\g_{2,r}^*\subseteq \g_2^*$. Moreover,
\[
|\mu|^{-1}\atL{b^{\mu}}{r}{d_1} \phi = \Big(\atL{b^{\bar\mu}}{r}{d_1}\big(\phi(|\mu|^{-\frac 1 2}\,\cdot\,)\big)\Big)(|\mu|^{\frac 1 2}\,\cdot\,),
\quad \text{where}\quad \bar\mu = \frac{\mu}{|\mu|},
\]
and thus
\begin{equation}\label{eq:rescaling-iii}
\mathbf{1}_{[K|\mu|,(K+1)|\mu|)}(\atL{b^\mu}{r}{d_1}) \phi = \Big(\mathbf{1}_{[K,K+1)}(\atL{b^{\bar\mu}}{r}{d_1})\big(\phi(|\mu|^{-\frac 1 2}\,\cdot\,)\big) \Big)(|\mu|^{\frac 1 2}\,\cdot\,).
\end{equation}
As stated in \cref{prop:rotation}, the function $\mu\mapsto R_\mu\in O(d_1)$ is in particular measurable. Hence, using \cref{eq:rotation-iii} and \cref{eq:rescaling-iii}, we obtain
\begin{align*}
\norm*{\mathbf{1}_{[K|\mu|,(K+1)|\mu|)}(L^\mu) }_{p\to 2}
& = \big\Vert\mathbf{1}_{[K|\mu|,(K+1)|\mu|)}(\atL{b^\mu}{r}{d_1}) \big\Vert_{p\to 2} \\
& = |\mu|^{\frac{d_1}2(\frac 1 p - \frac 1 2)}\big\Vert\mathbf{1}_{[K,K+1)}(\atL{b^{\bar\mu}}{r}{d_1})\big\Vert_{p\to 2}
\end{align*}
for almost all $\mu\in\g_{2,r}^*$. Let $S=\{\mu\in\g_2^*: |\mu|=1\}$ and $\textbf{B}$ be the image of $S$ under the map $\mu\mapsto \textbf{b}^\mu$. Since this map is continuous, the parameter set $\textbf{B}\subseteq [0,\infty)^N$ is compact. Hence we may apply \cref{thm:cluster} and get
\[
\big\Vert\mathbf{1}_{[K,K+1)}(\atL{b^{\bar\mu}}{r}{d_1}) \big\Vert_{p\to 2} \lesssim_{\mathbf B,\mathbf r,p} (K+1)^{\frac{d_1}2(\frac 1 p - \frac 1 2)-\frac 1 2}
\]
for all $\mu\in\g_2^*$, which yields \cref{eq:cluster-rescaled}.
\end{proof}

With the spectral cluster estimate \cref{eq:cluster-rescaled} we now prove the restriction type estimate \cref{eq:restriction-type-i}, which in combination with \cref{lem:restriction-type-p=1} yields \cref{thm:restriction-type}.

\begin{proof}[Proof of \cref{eq:restriction-type-i}]
Suppose that $1\le p\le \min\{p_{d_1},p_{d_2}\}$. We have to show \cref{eq:restriction-type-i} for $F:\R\to\C$ and $\chi:(0,\infty)\to\C$ being smooth cut-off functions, with $F$ being supported in the fixed compact set $A\subseteq (0,\infty)$ of \cref{thm:restriction-type}. Let $f$ be a Schwartz function on $G$. Using \cref{eq:joint-calculus-ii} and the Plancherel theorem on $L^2(\g_2)$, we obtain
\begin{align}
\norm*{ F(L)\chi(2^\ell U) f }^2_{L^2(G)}
& \sim \int_{\g_2^*} {\norm*{F(L^\mu)\chi(2^\ell |\mu|) f^\mu}_{L^2(\g_1)}^2} \, d\mu \notag \\
& \lesssim \int_{|\mu|\sim 2^{-\ell}} {\norm*{F(L^\mu) f^\mu}_{L^2(\g_1)}^2} \, d\mu. \label{eq:proof-restriction-i}
\end{align}
Moreover, orthogonality on $L^2(\g_1)$ yields
\begin{align}
& \norm*{F(L^\mu) f^\mu}_{L^2(\g_1)}^2
 = \sum_{K=0}^\infty {\norm*{F|_{[K|\mu|,(K+1)|\mu|)}(L^\mu) f^\mu}_{L^2(\g_1)}^2} \notag \\
& \le \sum_{K=0}^\infty {\norm*{F|_{[K|\mu|,(K+1)|\mu|)}}_\infty^2}\,  \norm*{\mathbf{1}_{[K|\mu|,(K+1)|\mu|)}(L^\mu) f^\mu}_{L^2(\g_1)}^2. \label{eq:proof-restriction-ii}
\end{align}
We may assume that $K|\mu|\lesssim_A 1$ and $(K+1)|\mu|\gtrsim_A 1$ since $F|_{[K|\mu|,(K+1)|\mu|)}=0$ for $\left[K|\mu|,(K+1)|\mu|\right) \cap A = \emptyset$. Then, for $|\mu|\sim 2^\ell$, we have $K\lesssim_A |\mu|^{-1}\sim 2^\ell$, and the spectral cluster estimate \cref{eq:cluster-rescaled} implies
\begin{align}
\norm*{\mathbf{1}_{[K|\mu|,(K+1)|\mu|)}(L^\mu) f^\mu}_{L^2(\g_1)}
 & \lesssim \left|\mu\right|^{\frac{d_1}2(\frac 1 p - \frac 1 2)}\left(K+1\right)^{\frac{d_1}2(\frac 1 p - \frac 1 2)-\frac 1 2} \left\| f^\mu \right\|_{L^p(\g_1)}  \notag \\
 & \lesssim \left|\mu\right|^{\frac 1 2} \left\| f^\mu \right\|_{L^p(\g_1)}. \label{eq:proof-restriction-iii}
\end{align}
Moreover, for $|\mu|\sim 2^{-\ell}$, we have
\[
|\mu| \sum_{K=0}^\infty {\norm*{F|_{[K|\mu|,(K+1)|\mu|)}}_\infty^2} \sim \norm*{F}_{2^{\ell},2}^2.
\]
Combining this estimate with \cref{eq:proof-restriction-i}, \cref{eq:proof-restriction-ii} and \cref{eq:proof-restriction-iii}, we get
\begin{equation}\label{eq:proof-restriction-v}
\norm*{ F(L)\chi(2^\ell U) f }^2_{L^2(G)} \lesssim \norm*{F}_{2^{\ell},2}^2 \int_{|\mu|\sim 2^{-\ell}}  {\norm*{f^\mu}_{L^p(\g_1)}^2}\,d\mu .
\end{equation}
Since $2/p\ge 1$, Minkowski's integral inequality yields
\begin{equation}\label{eq:proof-restriction-vi}
\int_{|\mu|\sim 2^{-\ell}} {\norm*{f^\mu}_{L^p(\g_1)}^2}\,d\mu
 \le \bigg(\int_{\g_1}\bigg( \int_{|\mu|\sim 2^{-\ell}} |f^\mu(x)|^2 \,d\mu\bigg)^{\frac p 2}\,dx \bigg)^{\frac 2 p}.
\end{equation}
Let $f_{x}:=f(x,\cdot)$ and $\widehat\cdot$ denote the Fourier transform on $\g_2$. Using polar coordinates and applying the Stein--Tomas restriction estimate \cite{To79} for the Euclidean sphere, we get
\begin{align}
\int_{|\mu|\sim 2^{-\ell}} |f^\mu(x)|^2 \,d\mu
 & = \int_{r\sim 2^{-\ell}} \int_{S^{d_2-1}} | \widehat{f_{x}}(r\omega)|^2 \, d\sigma(\omega) \, r^{d_2-1} \, dr\notag\\
 & = \int_{r\sim 2^{-\ell}} \int_{S^{d_2-1}} \big| \big(f_{x}(r^{-1}\,\cdot\,)\big)^\wedge(\omega)\big|^2 \, d\sigma(\omega)\, r^{-d_2-1} \, dr \notag \\
 & \lesssim \int_{r\sim 2^{-\ell}} {\norm*{f_{x}(r^{-1}\,\cdot\,)}_{L^p(\g_2)}^2\,r^{-d_2-1}} \, dr\notag \\
 & = \int_{r\sim 2^{-\ell}} r^{2d_2(\frac 1 p -\frac 1 2)-1}  \, dr \, \norm*{f_{x}}_{L^p(\g_2)}^2\notag \\
 & \sim 2^{-2\ell d_2(\frac 1 p -\frac 1 2)} \norm*{f_{x}}_{L^p(\g_2)}^2.\notag 
\end{align}
In combination with \cref{eq:proof-restriction-v} and \eqref{eq:proof-restriction-vi}, we obtain
\[
\norm*{ F(L)\chi(2^\ell U) f}_{L^2(G)}
\lesssim_{A,\chi} 2^{-\ell d_2(\frac 1p - \frac 1 2)} \norm{F}_{2^{\ell},2} \, \norm{f}_{L^p(G)},
\]
which yields \cref{eq:restriction-type-i}.
\end{proof}


\end{document}